\documentclass[10pt]{article}

\usepackage{amsmath,amsthm,amsfonts,amssymb,bbm}
\usepackage{graphicx}
\usepackage{pb-diagram}
\usepackage{psfig}
\usepackage[usenames,dvipsnames]{color}
\usepackage[hypertex]{hyperref}
\usepackage{soul}

\newcommand{\be}{\begin{enumerate}}
\newcommand{\ee}{\end{enumerate}}
\newcommand{\beq}{\begin{equation}}
\newcommand{\eeq}{\end{equation}}

\title{Groups acting freely on $\Lambda$-trees}
\author{\textsf{Olga Kharlampovich}\thanks{Supported by NSERC grant}
\and \textsf{Alexei Myasnikov}\thanks{Supported by NSERC grant}
\and \textsf{Denis Serbin}}

\date{\textsf{September 23, 2011}}

\newtheorem{cor}{Corollary}
\newtheorem{proposition}{Proposition}
\newtheorem{theorem}{Theorem}
\newtheorem{lemma}{Lemma}
\newtheorem{definition}{Definition}
\newtheorem{remark}{Remark}

\newtheorem{con}{Conjecture}

\begin{document}

\maketitle

\begin{abstract}
A group is called $\Lambda$-free if it has a free Lyndon length function in an ordered abelian group
$\Lambda$, which is equivalent to having a free isometric action on a $\Lambda$-tree. A group has a
regular free length function in $\Lambda$ if and only if it has a free isometric action on a
$\Lambda$-tree so that all branch points belong to the orbit of the base point. In this paper we prove
that every finitely presented $\Lambda$-free group $G$ can be embedded into a finitely presented group
with a regular free length function in $\Lambda$ so that the length function on $G$ is preserved by
the embedding. Next, we prove that every finitely presented group $\widetilde G$ with a regular free
Lyndon length function in $\Lambda$ has a regular free Lyndon length function in ${\mathbb R}^n$
ordered lexicographically for an appropriate $n$ and can be obtained from a free group by a series
of finitely many HNN-extensions in which associated subgroups are maximal abelian and length
isomorphic.

\end{abstract}

\section{Introduction}
\label{sec:1}

This is the fourth paper in a series, where we investigate non-Archimedean group actions, length
functions and infinite words. In \cite{L} Lyndon introduced real-valued length functions as a tool
to carry over Nielsen cancelation theory from free groups to a more general setting (see also
\cite{LS}).  Some results in this direction were obtained in \cite{Hoare1,Hoare2,Harrison,Prom, AM}.
In \cite{Ch} Chiswell described a crucial construction which shows that a group with a real-valued
length function has an action on an $\mathbb{R}$-tree, and vice versa. Later, Morgan and Shalen
realized that a similar construction holds for an arbitrary group with a Lyndon length function which
takes values in an arbitrary ordered abelian group $\Lambda$ (see \cite{MS}). In particular, they
introduced $\Lambda$-trees as a natural generalization of $\mathbb{R}$-trees which they studied in
relation with Thurston's Geometrization Program. Thus, actions on $\Lambda$-trees and Lyndon length
functions with values in $\Lambda$ are two equivalent languages describing the same class of groups.
In the case when the action is free (the stabilizer of every point is trivial) we call groups in this
class $\Lambda$-free. We refer to the book \cite{Ch1} for a detailed discussion on the subject.

One of the major events in combinatorial group theory in 1970's was the development of Bass-Serre
theory. We refer to the book \cite{Serre}, where Serre laid down fundamentals of the theory of
groups acting freely on simplicial trees. In particular, Bass-Serre theory makes it possible to
extract information about the structure of a group from its action on a simplicial tree.
Alperin and Bass \cite{AB} developed the initial framework of the theory of group actions on
$\Lambda$-trees and stated the fundamental research goals: find the group theoretic information
carried by an action (by isometries) on a $\Lambda$-tree; generalize Bass-Serre theory to actions
on arbitrary $\Lambda$-trees.

A joint effort of several researchers culminated in a description of finitely generated groups acting
freely on $\mathbb{R}$-trees \cite{BF,GLP}, which is now known as Rips' theorem: a finitely
generated group acts freely on an $\mathbb{R}$-tree if and only if it is a free product of free
abelian groups and surface groups (with an exception of non-orientable surfaces of genus $1, 2$, and
$3$). The key ingredient of this theory is the so-called ``Rips machine'', the idea of which comes
from Makanin's algorithm for solving equations in free groups (see \cite{Mak}). The Rips machine
appears in applications as a general tool that takes a sequence of isometric actions of a group $G$
on some ``negatively curved spaces'' and produces an isometric action of $G$ on an $\mathbb{R}$-tree
as the Gromov-Hausdorff limit of the sequence of spaces. Free actions on $\mathbb{R}$-trees cover
all Archimedean actions, since every group acting freely on a $\Lambda$-tree for an Archimedean ordered
abelian group $\Lambda$ also acts freely on an $\mathbb{R}$-tree.

In the non-Archimedean case there are only partial results for particular choices of $\Lambda$.
First of all, in \cite{B} Bass studied finitely generated groups acting freely on $\Lambda_0 \oplus
\mathbb{Z}$-trees with respect to the right lexicographic order on $\Lambda_0 \oplus \mathbb{Z}$. In
this case it was shown that the group acting freely on a $\Lambda_0 \oplus \mathbb{Z}$-tree splits
into a graph of groups with $\Lambda_0$-free vertex groups and maximal abelian edge groups. Next,
Guirardel (see \cite{G}) obtained the structure of finitely generated groups acting freely on
$\mathbb{R}^n$-trees (with the lexicographic order). In \cite{KMRS1} the authors described the
class of finitely generated groups acting freely and regularly on $\mathbb{Z}^n$-trees in terms of
HNN-extensions of a very particular type. The importance of regular actions becomes clear from the
results of \cite{KMS2}, where we proved that a finitely generated group acting freely on a
$\mathbb{Z}^n$-tree is a subgroup of a finitely generated group acting freely and regularly on a
$\mathbb{Z}^m$-tree for $m \geqslant n$, and the paper \cite{Ch2}, where is was shown that a group
acting freely on a $\Lambda$-tree (for arbitrary $\Lambda$) can always be embedded in a
length-preserving way into a group acting freely and regularly on a $\Lambda$-tree.

In this paper we give a partial solution (for finitely presented groups) of the following main
problem of the Alperin-Bass program.

\medskip

{\bf Problem.}
{\it Describe finitely presented (finitely generated) $\Lambda$-free groups for an arbitrary ordered
abelian group $\Lambda$.}

\medskip

Nowadays, the geometric method of group actions seems much more in use  than the length functions,
even though the Nielsen theory still  gives very powerful results (see the discussion in
\cite{KMRS1}). There are several reasons for this. Firstly, the Lyndon's abstract axiomatic approach
to length functions is less intuitive than group actions. Secondly, the current development of the
abstract Nielsen theory is incomplete and insufficient --- one of the principal notions of complete
actions and regular length functions were not in use until recently. Thirdly, and this is crucial,
an abstract analog of the Rips machine for length functions with values in $\mathbb{R}$ (or an
arbitrary $\Lambda$) was not introduced or developed enough to be used easily in applications.
Notice, that in the case of $\Lambda = \mathbb{Z}$ the completeness and regularity come for free,
so there was no need  for axiomatic formalization and it went mostly unnoticed. The regularity
axiom  appeared first in \cite{MR,MRS} as a tool to deal with length functions in $\mathbb{Z}^n$
(with respect to the lexicographic order). On the other hand, an analog of the Rips machine for
$\Lambda = \mathbb{Z}$ did exist for a long time --- the original Makanin-Razborov process
\cite{Raz} for solving equations in free groups. But it was not recognized as such until
\cite{KMIrc,KM2}, where it was used systematically to get splittings of groups.

Introduction of infinite $\Lambda$-words in \cite{MRS} gives another approach to group actions.
Every group admitting a faithful representation by $\Lambda$-words acts freely on a $\Lambda$-tree.
The converse is proved in \cite{Ch}. Formulation of the regularity axiom is another important
ingredient. This axiom ensures that the group is complete with respect to the Gromov's inner
product, that is, the Gromov's product of two elements of the group can be realized as the length
of a particular element of the group (not only as an abstract length). In the language of actions
this means that all branch points belong to the orbit of the base point. This allows one to use
the Nielsen cancelation argument, very similar to the case of a free group. The regularity
(completeness) condition is crucial for the existence of Makanin-Razborov type processes over
groups with $\Lambda$-length functions.

In this paper we prove the following results.

\begin{theorem}
\label{th:main1}
Any finitely presented regular $\Lambda$-free group $G$  can be represented as a union of a finite
series of groups
$$G_1 < G_2 < \cdots < G_n = G,$$
where
\begin{enumerate}
\item $G_1$ is a free group,
\item $G_{i+1}$ is obtained from $G_i$ by finitely many HNN-extensions in which associated subgroups
are maximal abelian and the associated isomorphisms preserve the length induced from $G_i$.
\end{enumerate}
\end{theorem}

\begin{theorem}
Let $G$ be a finitely presented regular $\Lambda$-free group. Then $G$ is a torsion-free toral
relatively hyperbolic biautomatic group with a quasi-convex hierarchy (see  \cite{wise} for the definition of a quasi-convex hierarchy).
\end{theorem}
\begin{proof}
It follows from the description of the structure of $G$ that each $G_i$ is toral relatively
hyperbolic. (It also follows from \cite{G} and Theorem 3 since $G$ acts  freely on
$\mathbb{R}^n$-tree.) From \cite{Rebeca} it follows that $G$ is biautomatic. Now by \cite{bridson}
centralizers are regular and quasi-convex in biautomatic groups. So the associated subgroups in
$G_i$ are quasiconvex in $G_i$. If they are not conjugate, then they are centralizers in $G_{i+1}$ and therefore
quasiconvex there. If they are conjugate, then the HNN-extension is equivalent to an extension of the centralizer of one of them (denote it by $A$).
The centralizer $C=<A, t| [a,t]=1, a\in A>$ is quasiconvex in $G_{i+1}$ and  $A$ is a direct summand of $C$, therefore $A$ is quasiconvex in $C$ and in $G_{i+1}.$
\end{proof}

\begin{cor}
Let $G$ be a finitely presented regular $\Lambda$-free group. Then the following algorithmic
problems are decidable in $G$:
\begin{itemize}
\item the Word and Conjugacy Problems;,
\item the Diophantine Problem (decidability of arbitrary equations in $G$).
\end{itemize}
\end{cor}

Indeed, decidability of equations follows from \cite{Dah}. Results of Dahmani and Groves \cite{DG}
imply the following two corollaries.

\begin{cor}
Let $G$ be a finitely presented regular $\Lambda$-free group. Then:
\begin{itemize}
\item $G$ has a non-trivial abelian splitting and one can find such a splitting effectively,
\item $G$ has a non-trivial abelian JSJ-decomposition and one can find such a decomposition
effectively.
\end{itemize}
\end{cor}

\begin{cor}
The isomorphism problem is decidable in the class of finitely presented groups that act regularly
and freely on some $\Lambda$-tree.
\end{cor}

\begin{theorem}
\label{th:main3}
Any finitely presented regular $\Lambda$-free group $G$ is $\mathbb{R}^n$-free for an appropriate
$n \in \mathbb{N}$, where $\mathbb{R}^n$ is ordered lexicographically.
\end{theorem}

\begin{theorem}
\label{th:main4}
Any finitely presented $\Lambda$-free group can be isometrically embedded into a finitely presented
regular $\Lambda$-free group.
\end{theorem}

\begin{cor}
Any finitely presented $\Lambda$-free group $G$ is $\mathbb{R}^n$-free for an appropriate $n \in
\mathbb{N}$ and lexicographically ordered $\mathbb{R}^n$.
\end{cor}

Our main conjecture about $\Lambda$-free groups is still open.

\begin{con}
\label{th:main2}
Any finitely presented $\Lambda$-free group $G$ is $\mathbb{Z}^k$-free for an appropriate $k \in
\mathbb{N}$ and lexicographically ordered $\mathbb{Z}^k$.
\end{con}

\section{Length functions, actions, infinite words}
\label{sec:prelim}

Here we introduce basic definitions and notations which are to be used throughout the whole paper.

\subsection{Lyndon length functions and free actions}
\label{subsec:Lyndon}

Let $G$ be a group and $\Lambda$ an ordered
abelian group. Then a function $l: G \rightarrow \Lambda$ is called a {\it (Lyndon) length function}
on $G$ if the following conditions hold:
\begin{enumerate}
\item [(L1)] $\forall\ g \in G:\ l(g) \geqslant 0$ and $l(1) = 0$;
\item [(L2)] $\forall\ g \in G:\ l(g) = l(g^{-1})$;
\item [(L3)] $\forall\ g, f, h \in G:\ c(g,f) > c(g,h)
\rightarrow c(g,h) = c(f,h)$,

\noindent where $c(g,f) = \frac{1}{2}(l(g)+l(f)-l(g^{-1}f))$.
\end{enumerate}
Sometimes we refer to length functions with values in $\Lambda$ as to $\Lambda$-length functions.
Observe, that the word length function on a finitely generated group (with respect to a fixed
finite generating set) usually is not a Lyndon length function. In this paper we consider length
functions only of Lyndon type.

Notice that  $c(g,f)$ may not be defined in $\Lambda$ (if $l(g) + l(f) - l(g^{-1}f)$ is not divisible by
$2$), so in the axiom (L3) we assume that $\Lambda$ is canonically embedded into a divisible ordered
abelian group $\Lambda_{\mathbb{Q}} = \Lambda \otimes_\mathbb{Z} \mathbb{Q}$ (see \cite{MRS} for
details).

It is not difficult to derive the following two properties of length functions from the axioms (L1)--
(L3):
\begin{itemize}
\item $\forall\ g, f \in G:\ l(g f) \leqslant l(g) + l(f)$;
\item $\forall\ g, f \in G:\ 0 \leqslant c(g,f) \leqslant \min\{l(g),l(f)\}$.
\end{itemize}
A length function $l : G \rightarrow \Lambda$ is called {\em free} if it satisfies the following two
axioms.

\begin{enumerate}
\item [(L4)] $\forall\ g, f \in G:\ c(g,f) \in \Lambda.$
\item [(L5)] $\forall\ g \in G:\ g \neq 1 \rightarrow l(g^2) > l(g).$
\end{enumerate}

As we have mentioned in the introduction group actions on $\Lambda$-trees provide a geometric
counterpart to $\Lambda$-length functions. To explain we need the following definitions, which we
also use later in the text.

Let $X$ be a non-empty set and $\Lambda$ an ordered abelian group. A {\em $\Lambda$-metric on $X$}
is a mapping $p: X \times X \longrightarrow \Lambda$ such that for all $x,y,z \in X$:
\begin{enumerate}
\item[(M1)] $p(x,y) \geqslant 0$ and $p(x,y) = 0$ if and only if $x = y$,
\item[(M2)] $p(x,y) = p(y,x)$,
\item[(M3)] $p(x,y) \leqslant p(x,z) + p(y,z)$.
\end{enumerate}
If the axioms (M1)--(M3) are satisfied then the pair $(X,p)$ is an {\em $\Lambda$-metric space}. If
$(X,p)$ and $(X',p')$ are $\Lambda$-metric spaces, an {\it isometry} from $(X,p)$ to $(X',p')$ is a
mapping $f: X \rightarrow X'$ such that $p(x,y) = p'(f(x),f(y))$ for all $x,y \in X$.

For elements $a,b \in \Lambda$  the {\it closed segment} $[a,b]$ is defined by
$$[a,b] = \{c \in A \mid a \leqslant c \leqslant b \}.$$
More generally, a {\it segment} in an $\Lambda$-metric space is the image of an isometry $\alpha:
[a,b] \rightarrow X$ for some $a,b \in \Lambda$. The endpoints of the segment are $\alpha(a),
\alpha(b)$. A segment with endpoints $x, y \in X$ is denoted by $[x,y]$. An $\Lambda$-metric space
$(X,p)$ is {\it geodesic} if for all $x,y \in X$, there is a segment in $X$ with endpoints $x,y$.

An {\em $\Lambda$-tree} is an $\Lambda$-metric space $(X,p)$ such that:
\begin{enumerate}
\item[(T1)] $(X,p)$ is geodesic,
\item[(T2)] if two  segments of $(X,p)$ intersect in a single point, which is an endpoint of both,
then their union is a segment,
\item[(T3)] the intersection of two segments with a common endpoint is also a segment.
\end{enumerate}

We say that a group $G$ acts on a $\Lambda$-tree $X$ if it acts on $X$ by isometries, i.e. there
exists an embedding of $G$ into $Isom(X)$. An action of $G$ on $X$ is termed {\em free} if for every
$1 \neq g \in G$ neither $g$, nor $g^2$ has a fixed point in $X$.

If $G$ acts on a $\Lambda$-tree $(X,p)$ and $x \in X$ then one can associate to this action a
function $l_x : G \to \Lambda$ so that $l_x(g) = p(x, g x)$ for every $g \in G$. It is not hard to
check that $l_x$ is a Lyndon length function. Moreover, if the action of $G$ on $X$ is free then
$l_x$ is a free length function. The converse is also true, that is, every action of a group $G$ on
a $\Lambda$-tree arises from some Lyndon $\Lambda$-length function.

\begin{theorem}
\cite{Ch, MS}
Let $G$ be a group and $l: G \rightarrow \Lambda$ a Lyndon length function satisfying (L4). Then
there are a $\Lambda$-tree $(X,p)$, an action of $G$ on $X$ and a point $x \in X$ such that $l =
l_x$.
\end{theorem}

\smallskip

For elements $g_1, \ldots, g_n \in G$ we write
$$g = g_1 \circ \cdots \circ g_n$$
if $g = g_1 \cdots g_n$ and $l(g) = l(g_1) + \cdots + l(g_n)$. Also, for $\alpha \in \Lambda$ we
write $g = g_1 \circ_\alpha g_2$ if $g = g_1 g_2$ and $c(g_1^{-1},g_2) < \alpha$.

A length function $l: G \rightarrow \Lambda$ is called {\it regular} if it satisfies the {\it
regularity} axiom:
\begin{enumerate}
\item [(L6)] $\forall\ g, f \in G,\ \exists\ u, g_1, f_1 \in G$:
$$g = u \circ g_1 \ \& \  f = u \circ f_1 \ \& \ l(u) = c(g,f).$$
\end{enumerate}

Many examples of groups with regular free length functions are given in \cite{KMRS1}.

\subsection{Regular actions}
\label{sec:regact}

In this subsection we give a geometric characterization of group actions that come from regular
length functions. The four lemmas below were proved in \cite{KMRS1}, but for convenience of the
reader we give the proofs here too.

\begin{definition}
Let $G$ act on a $\Lambda$-tree $\Gamma$. The action is regular with respect to $x \in \Gamma$ if
for any $g,h \in G$ there exists $f \in G$ such that $[x, f x] = [x, g x] \cap [x, h x]$.
\end{definition}

\begin{lemma}
\cite{KMRS1}
Let $G$ act on a $\Lambda$-tree $(\Gamma,d)$. Then the action of $G$ is regular with respect to $x
\in \Gamma$ if and only if the length function $l_x: G \rightarrow \Lambda$ based at $x$ is regular.
\end{lemma}
\begin{proof} By definition, the length function $l_x$ is regular if for every $g, h \in G$ there
exists $f \in G$ such that $g = f g_1,\ h = f h_1$, where $l_x(f) = c(g, h)$ and $l_x(g) = l_x(f) +
l_x(g_1),\ l_x(h) = l_x(f) + l_x(h_1)$.

\smallskip

Suppose the action of $G$ is regular with respect to $x$. Then for $g,h \in G$ there exists $f \in
G$ such that $[x, f x] = [x, g x] \cap [x, h x]$. We have $[x, g x] = [x, f x] \cup [f x, g x],\ [x,
h x] = [x, f x] \cup [f x, h x]$ and $[f x, g x] = [x, (f^{-1} g)x] = l_x(f^{-1} g),\ [f x, h x] =
[x, (f^{-1} h) x] = l_x(f^{-1} h)$. Taking $g_1 = f^{-1} g,\ h_1 = f^{-1} h$ we have $l_x(g) = l_x(f)
+ l_x(g_1),\ l_x(h) = l_x(f) + l_x(h_1)$. Finally, since $c(g,h) = \frac{1}{2}(l_x(g) + l_x(h) -
l_x(g^{-1} h))$ and $l_x(g^{-1} h) = d(x, (g^{-1}h)x) = d(g x, h x) = d(f x, g x) + d(f x, h x)$ we
get $l_x(f) = c(g,h)$.

\smallskip

Suppose that $l_x$ is regular. Then from $g = f \circ g_1,\ h = f \circ h_1$ it follows that $[x,
g x] = [x, f x] \cup [f x, g x],\ [x, h x] = [x, f x] \cup [f x, h x]$. $l_x(f) = c(g,h) =
\frac{1}{2}(l_x(g) + l_x(h) - l_x(g^{-1} h))$, so $2d(x, f x) = d(x, g x) + d(x, h x) - d(x, (g^{-1}
h)x) = d(x, g x) + d(x, h x) - d(g x, h x)$. In other words,
$$d(g x, h x) = d(x, g x) + d(x, h x) - 2d(x, f x) = (d(x, g x) - d(x, f x)) + (d(x, h x) - d(x, f x))$$
$$ = d(f x, g x) + d(f x, h x)$$
which is equivalent to $[x, f x] = [x, g x] \cap [x, h x]$.
\end{proof}

\begin{lemma}
Let $G$ act minimally on a $\Lambda$-tree $\Gamma$. If the action of $G$ is regular with respect to
$x \in \Gamma$ then all branch points of $\Gamma$ are $G$-equivalent.
\end{lemma}
\begin{proof} From minimality of the action it follows that $\Gamma$ is spanned by the set of points
$G x = \{g x \mid g \in G\}$.

Now let $y$ be a branch point in $\Gamma$. It follows that there exist (not unique in general) $g,h
\in G$ such that $[x,y] = [x, g x] \cap [x, h x]$. From regularity of the action it follows that
there exists $f \in G$ such that $y = f x$. Hence, every branch point is $G$-equivalent to $x$ and
the statement of the lemma follows.
\end{proof}

\begin{lemma}
Let $G$ act on a $\Lambda$-tree $\Gamma$. If the action of $G$ is regular with respect to $x \in
\Gamma$ then it is regular with respect to any $y \in G x$.
\end{lemma}
\begin{proof} We have to show that for every $g, h \in G$ there exists $f \in G$ such that
$[y, f y] = [y, g y] \cap [y, h y]$. Since $y = t x$ for some $t \in G$ then we have to prove that
$[t x, (f t) x] = [t x, (gt) x] \cap [t x, (ht) x]$. The latter equality is equivalent to
$[x, (t^{-1} f t) x] = [x, (t^{-1} g t) x] \cap [x, (t^{-1} h t) x]$ which follows from regularity
of the action with respect to $x$.
\end{proof}

\begin{lemma}
Let $G$ act freely on a $\Lambda$-tree $\Gamma$ so that all branch points of $\Gamma$ are
$G$-equivalent. Then the action of $G$ is regular with respect to any branch point in $\Gamma$.
\end{lemma}
\begin{proof} Let $x$ be a branch point in $\Gamma$ and $g, h \in G$. If $g = h$ then $[x, g x] \cap
[x, h x] = [x, g x]$ and $g$ is the required element. Suppose $g \neq h$. Since the action is free
then $g x \neq h x$ and we consider the tripod formed by $x, gx, hx$. Hence, $y = Y(x, gx, hx)$ is a
branch point in $\Gamma$ and by the assumption there exists $f \in G$ such that $y = f x$.
\end{proof}

\subsection{Infinite words and length functions}
\label{subs:2.1}

In this subsection at first we recall some notions from the theory of ordered abelian groups (for
all the details refer to the books \cite{Gl} and \cite{KopMed}) and then following \cite{MRS}
describe the construction of infinite words.

\smallskip

Let $\Lambda$ be a discretely ordered abelian group with the minimal positive element $1$. It is
going to be clear from the context if we are using $1$ as an element of $\Lambda$, or as an integer.
Let $X = \{x_i \mid i \in I\}$ be a set. Put $X^{-1} = \{x_i^{-1} \mid i \in I\}$ and $X^\pm = X
\cup X^{-1}$. A {\em $\Lambda$-word} is a function of the type
$$w: [1,\alpha_w] \to X^\pm,$$
where $\alpha_w \in \Lambda,\ \alpha_w \geqslant 0$. The element $\alpha_w$ is called the {\em
length} $|w|$ of $w$.

\smallskip

By $W(\Lambda,X)$ we denote the set of all $\Lambda$-words. Observe, that $W(\Lambda,X)$ contains an
empty $\Lambda$-word which we denote by $\varepsilon$.

Concatenation $uv$ of two $\Lambda$-words $u,v \in W(\Lambda,X)$ is a $\Lambda$-word of length $|u|
+ |v|$ and such that:
\[ (uv)(a) = \left\{ \begin{array}{ll}
\mbox{$u(a)$}  & \mbox{if $1 \leqslant a \leqslant |u|$} \\
\mbox{$v(a - |u|)$ } & \mbox{if $|u|  < a \leqslant |u| + |v|$}
\end{array}
\right. \]

A $\Lambda$-word $w$ is {\it reduced} if $w(\beta + 1) \neq w(\beta)^{-1}$ for each $1 \leqslant
\beta < |w|$. We denote by $R(\Lambda,X)$ the set of all reduced $\Lambda$-words. Clearly,
$\varepsilon \in R(\Lambda,X)$. If the concatenation $uv$ of two reduced $\Lambda$-words $u$ and
$v$ is also reduced then we write $uv = u \circ v$.

\smallskip

For $u \in W(\Lambda,X)$ and $\beta \in [1, \alpha_u]$ by $u_\beta$ we denote the restriction of $u$
on $[1,\beta]$. If $u \in R(\Lambda,X)$ and $\beta \in [1, \alpha_u]$ then
$$u = u_\beta \circ {\tilde u}_\beta,$$
for some uniquely defined ${\tilde u}_\beta$.

An element ${\rm com}(u,v) \in R(\Lambda,X)$ is called the ({\emph{longest}) {\it common initial
segment} of $\Lambda$-words $u$ and $v$ if
$$u = {\rm com}(u,v) \circ \tilde{u}, \ \ v = {\rm com}(u,v) \circ \tilde{v}$$
for some (uniquely defined) $\Lambda$-words $\tilde{u}, \tilde{v}$ such that $\tilde{u}(1) \neq
\tilde{v}(1)$.

Now, we can define the product of two $\Lambda$-words. Let $u,v \in R(\Lambda,X)$. If
${\rm com}(u^{-1}, v)$ is defined then
$$u^{-1} = {\rm com}(u^{-1},v) \circ {\tilde u}, \ \ v = {\rm com} (u^{-1},v) \circ {\tilde v},$$
for some uniquely defined ${\tilde u}$ and ${\tilde v}$. In this event put
$$u \ast v = {\tilde u}^{-1} \circ {\tilde v}.$$
The  product ${\ast}$ is a partial binary operation on $R(\Lambda,X)$.

\smallskip

An element $v \in R(\Lambda,X)$ is termed {\it cyclically reduced} if $v(1)^{-1} \neq v(|v|)$. We
say that an element $v \in R(\Lambda,X)$ admits a {\it cyclic decomposition} if $v = c^{-1} \circ u
\circ c$, where $c, u \in R(\Lambda,X)$ and $u$ is cyclically reduced. Observe that a cyclic
decomposition is unique (whenever it exists). We denote by $CR(\Lambda,X)$ the set of all cyclically
reduced words in $R(\Lambda,X)$ and by $CDR(\Lambda,X)$ the set of all words from $R(\Lambda,X)$
which admit a cyclic decomposition.

Let $I_{\Lambda}$ index the set of all convex subgroups of ${\Lambda}$. $I_{\Lambda}$ is linearly
ordered (see, for example, \cite{Ch1}): $i < j$ if and only
if ${\Lambda}_i < {\Lambda}_j$, and
$${\Lambda} = \bigcup_{i\in I_{\Lambda}} {\Lambda}_i.$$
We say that $g \in G$ {\em has the height} $i \in I_{\Lambda}$ and denote
$ht(g) = i$ if $|g| \in {\Lambda}_i$ and $|g| \notin {\Lambda}_j$ for any $j < i$.

\smallskip

Below we refer to $\Lambda$-words as {\it infinite words} usually omitting $\Lambda$ whenever it
does not produce any ambiguity.

The following result establishes the connection between infinite words and length functions.
\begin{theorem}
\label{co:3.1} \cite{MRS}
Let $\Lambda$ be a discretely ordered abelian group and $X$ be a set. Then any subgroup $G$ of
$CDR(\Lambda,X)$ has a free Lyndon length function with values in $\Lambda$ -- the restriction
$L|_G$ on $G$ of the standard length function $L$ on $CDR(\Lambda,X)$.
\end{theorem}

The converse of the theorem above is also true.

\begin{theorem}
\label{chis} \cite{Ch3}
Let $G$ have a free Lyndon length function $L : G \rightarrow \Lambda$, where $\Lambda$ is a
discretely ordered abelian group. Then there exists a set $X$ and a length preserving embedding
$\phi : G \rightarrow CDR(\Lambda,X)$, that is, $|\phi(g)| = L(g)$ for any $g \in G$.
\end{theorem}

\begin{cor}
\label{chis-cor} \cite{Ch3}
Let $G$ have a free Lyndon length function $L : G \rightarrow \Lambda$, where $\Lambda$ is an
arbitrary ordered abelian group. Then there exists an embedding $\phi : G \rightarrow CDR(\Lambda',
X)$, where $\Lambda' = \mathbb{Z} \oplus \Lambda$ is discretely ordered with respect to the right
lexicographic order and $X$ is some set, such that, $|\phi(g)| = (0,L(g))$ for any $g \in G$.
\end{cor}

Theorem \ref{co:3.1}, Theorem \ref{chis}, and Corollary \ref{chis-cor} show that a group has a free
Lyndon length function if and only if it embeds into a set of infinite words and this embedding
preserves the length. Moreover, it is not hard to show that this embedding also preserves regularity
of the length function.

\begin{theorem}
\label{chis-cor-1} \cite{KhMS}
Let $G$ have a free regular Lyndon length function $L : G \rightarrow \Lambda$, where $\Lambda$ is
an arbitrary ordered abelian group. Then there exists an embedding $\phi : G \rightarrow R(\Lambda',
X)$, where $\Lambda'$ is a discretely ordered abelian group and $X$ is some set, such that, the
Lyndon length function on $\phi(G)$ induced from $R(\Lambda',X)$ is regular.
\end{theorem}

\section{Generalized equations}
\label{se:4-1}

From now on we assume that $G = \langle X \mid R \rangle$ is a finitely presented group which acts
freely and regularly on a $\Lambda$-tree, where $\Lambda$ is a discretely ordered abelian group, or,
equivalently, $G$ can be represented by $\Lambda$-words over some alphabet $Z$ and the length
function on $G$ induced from $CDR(\Lambda, Z)$ is regular. Let us fix the embedding $\xi : G
\hookrightarrow CDR(\Lambda ,Z)$ for the rest of this section.

\subsection{The notion of a generalized equation}
\label{subs:notion_gen_eq}

\begin{definition}
A {\em combinatorial generalized equation} $\Omega$ (which is convenient to visualize as shown on
the picture below)

\begin{figure}[htbp]
\label{pic:example}
\centering{\mbox{\psfig{figure=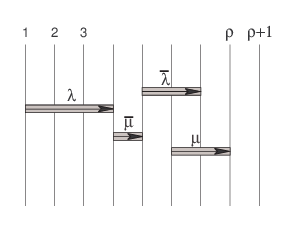}}}
\caption{A typical generalized equation.}
\end{figure}

consists of the following objects.

\begin{enumerate}

\item A finite set of {\em bases} ${\mathcal M} = BS(\Omega)$. The set of bases ${\mathcal M}$
consists of $2n$ elements ${\mathcal M} = \{\mu_1, \ldots, \mu_{2n}\}$. The set ${\mathcal M}$ comes
equipped with two functions: a function $\varepsilon: {\mathcal M} \rightarrow \{1,-1\}$ and an
involution $\Delta : {\mathcal M} \rightarrow {\mathcal M}$ (that is, $\Delta$ is a bijection such
that $\Delta^2$ is an identity on ${\mathcal M}$). Bases $\mu$ and $\overline\mu = \Delta(\mu)$ are
called {\it dual bases}. We denote bases by letters $\mu, \lambda$, etc.

\item A set of {\em boundaries} $BD = BD(\Omega) = \{1, 2, \ldots, \rho + 1\}$, that is, integer
points of the interval $I = [1, \rho + 1]$. We use letters $i,j$, etc. for boundaries.

\item Two functions $\alpha : BS \rightarrow BD$ and $\beta : BS \rightarrow BD$. We call
$\alpha(\mu)$ and $\beta(\mu)$ the {\em initial and terminal boundaries} of the base $\mu$ (or
endpoints of $\mu$). These functions satisfy the following conditions: for every base $\mu \in BS$:
$\alpha(\mu) < \beta(\mu)$ if $\varepsilon(\mu) = 1$ and $\alpha(\mu) > \beta(\mu)$ if
$\varepsilon(\mu) = -1$.

\item A set of {\em boundary connections} $(p,\lambda , q),$ where $p$ is a boundary on $\lambda$
(that is a number between $\alpha(\lambda)$ and $\beta(\lambda)$) and $q$ on $\bar\lambda$. In this
case we say that $p$ and $q$ are $\lambda$-tied. If $(p,\lambda ,q)$ is a boundary connection then
$(q,\overline\lambda , p)$ is also a boundary connection. (The meaning of boundary connections will
be explained in the transformation (ET5)).
\end{enumerate}
\end{definition}

With a combinatorial generalized equation $\Omega$ one can canonically associate a system of
equations in {\em variables} $h = (h_1, \ldots, h_\rho)$ (variables $h_i$ are also called {\it
items}). This system is called a {\em generalized equation}, and (slightly abusing the terminology)
we denote it by the same symbol $\Omega$, or $\Omega(h)$ specifying the variables it depends on. The
generalized equation $\Omega$ consists of the following two types of equations.

\begin{enumerate}
\item Each pair of dual  bases $(\lambda, \overline\lambda)$ provides an equation
$$[h_{\alpha(\lambda)} h_{\alpha(\lambda) + 1} \ldots h_{\beta(\lambda ) - 1}]^{\varepsilon
(\lambda)} = [h_{\alpha(\overline\lambda)} h_{\alpha(\overline\lambda) + 1} \ldots h_{\beta(\overline
\lambda ) - 1}]^{\varepsilon(\overline\lambda)}.$$
These equations are called {\em basic equations}.

\item Every boundary connection $(p,\lambda,q)$ gives rise to a {\em boundary equation}
$$[h_{\alpha(\lambda)} h_{\alpha(\lambda) + 1} \cdots h_{p-1}] = [h_{\alpha(\overline\lambda)}
h_{\alpha(\overline\lambda ) + 1} \cdots h_{q-1}],$$
if $\varepsilon(\lambda) = \varepsilon(\overline\lambda)$ and
$$[h_{\alpha (\lambda )} h_{\alpha(\lambda ) + 1} \cdots h_{p-1}] = [h_{q} h_{q+1} \cdots h_{\alpha (
\overline\lambda)-1}]^{-1} ,$$
if $\varepsilon(\lambda)= -\varepsilon(\overline\lambda).$
\end{enumerate}

\begin{remark}
We  assume that every generalized equation comes from a combinatorial one.
\end{remark}

Given a generalized equation $\Omega(h)$ one can define the {\em group of $\Omega(h)$}
$$G_\Omega = \langle h \mid \Omega(h) \rangle.$$

\begin{definition}
Let $\Omega(h) = \{L_1(h) = R_1(h), \ldots, L_s(h) = R_s(h)\}$ be a generalized equation in
variables $h = (h_1, \ldots, h_\rho)$. A set $U = (u_1, \ldots, u_\rho) \subseteq R(\Lambda,Z)$
of nonempty $\Lambda$-words is called a {\em solution} of $\Omega$ if:
\begin{enumerate}
\item all words $L_i(U), R_i(U)$ are reduced,
\item $L_i(U) =  R_i(U),\ i \in [1,s]$.
\end{enumerate}
\end{definition}

Observe that a solution $U$ of $\Omega(h)$ defines a homomorphism $\xi_U : G_\Omega \to R(\Lambda,Z)$
induced by the mapping $h_i \to u_i,\ i \in [1,\rho]$ since after this substitution all the equations
of $\Omega(h)$ turn into identities in $R(\Lambda,Z)$.

If we specify a particular solution $U$ of a generalized equation $\Omega$ then we use a pair
$(\Omega, U)$.

\begin{definition}
A cancelation table $C(U)$ of a solution $U = (u_1, \ldots, u_\rho)$ is defined as follows
$$C(U) = \{ h_i^\epsilon h_j^\sigma \mid {\rm there\ is\ cancelation\ in\ the\ product}\ u_i^\epsilon
\ast u_j^\sigma,\ {\rm where}\  \epsilon, \sigma = \pm 1 \}.$$
\end{definition}

\begin{definition}
A solution $U^+$ of a generalized equation $\Omega$ is called {\em consistent} with a solution $U$
if $C(U^+) \subseteq C(U)$.
\end{definition}

\subsection{From a finitely presented group to a generalized equation}
\label{subs:constr_ge}

Recall that $G = \langle X \mid R \rangle$ is finitely presented and let $X = \{x_1, \ldots, x_n\}$
and $R = \{r_1(X),\ldots, r_m(X)\}$. Adding, if necessary, auxiliary generators, we can assume that
every relator involves at most three generators.

Since $\xi$ is a homomorphism it follows that after the substitution $x_i \rightarrow \xi(x_i),\ i
\in [1,n]$ all products $r_i(\xi(X)),\ i \in [1,m]$ cancel out. Hence, we have finitely many {\em
cancellation diagrams} over $CDR(\Lambda ,Z)$, which give rise to a generalized equation $\Omega$
corresponding to the embedding $\xi : G \hookrightarrow CDR(\Lambda, Z)$.

The precise definition and all the details concerning cancellation diagrams over $CDR(\Lambda ,Z)$
can be found in \cite{KhMS}. Briefly, a cancellation diagram for $r_i(\xi(X))$ can be viewed as a
finite directed tree $T_i$ in which every positive edge $e$ has a label $\lambda_e$ so that every
occurrence $x^\delta,\ \delta \in \{-1,1\}$ of $x \in X$ in $r_i$ corresponds to a reduced path
$e_1^{\epsilon_1} \cdots e_k^{\epsilon_k}$, where $\epsilon_i \in \{-1,1\}$, in $T_i$ and
$\xi(x^\delta) = \lambda_{e_1}^{\epsilon_1} \circ \ldots \circ \lambda_{e_k}^{\epsilon_k}$. In other
words, each $\lambda_e$ is a piece of some generator of $G$ viewed as a $\Lambda$-word. Moreover,
we assume that $|\lambda_e|$ is known (since we know the homomorphism $\xi$).

\begin{figure}[htbp]
\label{triangle}
\centering{\mbox{\psfig{figure=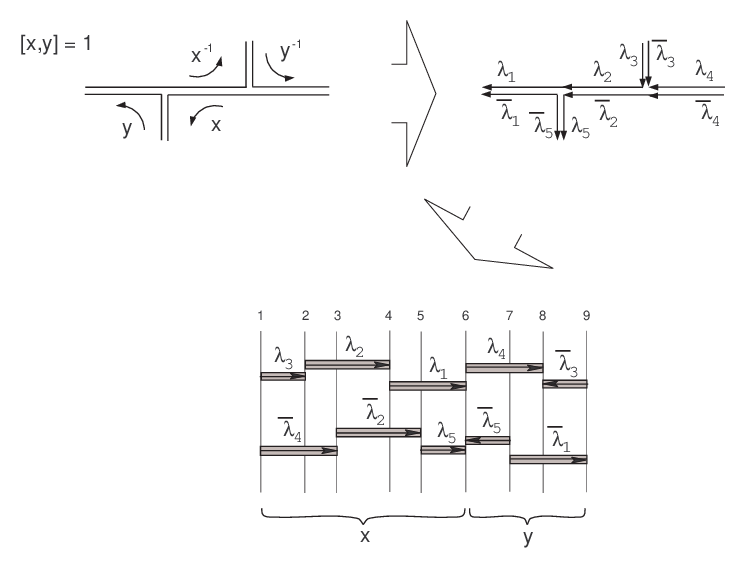}}}
\caption{From the cancellation diagram for the relation $[x,y] = 1$ to the generalized equation.}
\end{figure}

Now we would like to construct a generalized equation $\Omega_i$ corresponding to $T_i$.
Denote by $X(T_i)$ all generators of $G$ which appear in $r_i$. Next, consider a segment $J$ in
$\Lambda$ of length
$$\sum_{x \in X(T_i)} |\xi(x)|$$
which is naturally divided by the lengths of $\xi(x),\ x \in X(T_i)$ into subsegments with respect to
any given order on $X(T_i)$. Since every $\xi(x),\ x \in X(T_i)$ splits into at least one reduced
product $\lambda_{e_1}^{\epsilon_1} \circ \ldots \circ \lambda_{e_k}^{\epsilon_k}$, every such
splitting gives a subdivision of the corresponding subsegment of $J$. Hence, we subdivide $J$
using all product representations of all $\xi(x),\ x \in X(T_i)$. As a result we obtain a
subdivision of $J$ into $\rho_i$ items whose endpoints become boundaries of $\Omega_i$.
Observe that each $\lambda_e$ appears exactly twice in the products representing some $\xi(x),\ x
\in X(T_i)$ and each such entry covers several adjacent items of $J$. This pair of entries defines
a pair of dual bases $(\lambda_e, \overline{\lambda_e})$.
Hence,
$${\mathcal M}_i = BS(\Omega_i) = \{ \lambda_e,\ \overline{\lambda_e} \mid e \in E(T_i)\}.$$
$\epsilon(\lambda_e)$ depends on the sign of $\lambda_e$ in the corresponding product representing
a variable from $X(T_i)$ (similarly for $\overline{\lambda_e}$).

In the same way one can construct $T_i$ and the corresponding $\Omega_i$ for each $r_i,\ i \in [1,m]$.
Combining all combinatorial generalized equations $\Omega_i,\ i \in [1,m]$ we obtain the equation
$\Omega$ with items $h_1,\ldots,h_\rho$ and bases ${\mathcal M} = \cup_i {\mathcal M}_i$. By definition
$$G_\Omega = \langle h_1, \ldots, h_\rho \mid \Omega(h_1,\ldots,h_\rho)\rangle.$$
At the same time, since each item can be obtained in the form
$$(\lambda_{i_1}^{\epsilon_1} \circ \ldots \circ \lambda_{i_k}^{\epsilon_k}) \ast
(\lambda_{j_1}^{\delta_1} \circ \ldots \circ \lambda_{j_l}^{\delta_l})^{-1},$$
it follows that $G_\Omega$ can be generated by ${\mathcal M}$ with the relators obtained by
rewriting $\Omega(h_1,\ldots,h_\rho)$ in terms of ${\mathcal M}$.

It is possible to transform the presentation $\langle h_1, \ldots, h_\rho \mid \Omega\rangle$ into
$\langle X \mid R \rangle$ using Tietze transformations as follows. From the cancellation diagrams
constructed for each relator in $R$ it follows that $x_i = w_i(h_1,\ldots,h_\rho) = w_i(\overline{h}
),\ i \in [1,n]$. Hence
$$\langle h_1, \ldots, h_\rho \mid \Omega\rangle \simeq \langle h_1, \ldots, h_\rho, X \mid \Omega
\cup \{x_i = w_i(\overline{h}),\ i \in [1,n]\} \rangle.$$
Next, from the cancellation diagrams it follows that $R$ is a set of consequences of $\Omega \cup
\{x_i = w_i(\overline{h}),\ i \in [1,n]\}$, hence,
$$\langle h_1, \ldots, h_\rho \mid \Omega\rangle \simeq \langle h_1, \ldots, h_\rho, X \mid \Omega
\cup \{x_i = w_i(\overline{h}),\ i \in [1,n]\} \cup R \rangle.$$
Finally, since the length function on $G$ is regular, for each $h_i$ there exists a word $u_i(X)$
such that $h_i = u_i(\xi(X))$ and all the equations in $\Omega \cup \{x_i = w_i(\overline{h}),\ i
\in [1,n]\}$ follow from $R$ after we substitute $h_i$ by $u_i(X)$ for each $i$. It follows that
$$\langle h_1, \ldots, h_\rho, X \mid \Omega \cup \{x_i = w_i(\overline{h}),\ i \in [1,n]\} \cup R
\rangle$$
$$\simeq \langle h_1, \ldots, h_\rho, X \mid \Omega \cup \{x_i = w_i(\overline{h}),\ i \in [1,n]\}
\cup R \cup \{h_j = u_j(X),\ j \in [1,\rho]\} \rangle \simeq \langle X \mid R \rangle.$$
It follows that $G \simeq G_\Omega$.

Let $B$ be the standard CW $2$-complex for $G$ corresponding to the presentation $\langle {\mathcal M}
\mid \Omega({\mathcal M}) \rangle$. Let $C(\Omega)$ be a CW $2$-complex obtained from $G_{\Omega}$ by
creating for each relator a polygon with boundary labeled by this relator and gluing these polygons
together along edges labeled by the same letter. $C(\Omega)$ may have several vertices. If we add arcs
joining all the vertices in $C(\Omega)$ we obtain a CW complex $D(\Omega)$ which has the same
fundamental group as $B$ (namely, $G$).

\begin{lemma}
$G$ is a free product of $\pi_1(C(\Omega))$ and, possibly, a free group.
\end{lemma}
\begin{proof}
Follows from the construction.
\end{proof}

\begin{remark}
\label{nonreg}
Let $\widetilde G$ be a finitely presented group with a free length function in $\Lambda$ (not
necessary regular). It can be embedded isometrically in the group $\widehat G$ with a free regular
length function in $\Lambda$ by \cite{Ch2}. That group can be embedded in $R(\Lambda', X)$. When we
make a generalized equation $\Omega$ for $\widetilde G$, we have to add only finite number of
elements from $\widehat G$. Let $G$ be a subgroup generated in $\widehat G$ by $\widetilde G$ and
these elements. Then $G$ is the quotient of $G_{\Omega}$ containing $\widetilde G$ as a subgroup.
\end{remark}

\section{Elimination process}
\label{se:5}

In this section at first we introduce transformations of combinatorial generalized equations, then we
show how they can be applied to ``simplify'' the given equation $\Omega$ and to extract information
about the structure of the corresponding group $G_\Omega$.

\subsection{Elementary transformations}
\label{se:5.1}

In this subsection we describe {\em elementary transformations} of generalized equations. Let
$(\Omega, U)$ be a generalized equation together with a solution $U$. An elementary transformation
(ET) associates to a generalized equation $(\Omega, U)$ a generalized equation $(\Omega_1, U_1)$ and
an epimorphism $\pi : G_\Omega \rightarrow G_{\Omega_1}$ such that for the solution $U_1$ the
following diagram commutes

\[\begin{diagram}
\label{diag:1} \node{G_\Omega} \arrow{e,t}{\pi}
\arrow{s,l}{\xi_U} \node{G_{\Omega_1}}
\arrow{sw,r}{\xi_{U_1}}\\
\node{R(\Lambda,Z)}
\end{diagram}
\]

One can view (ET) as a mapping $ET : (\Omega,U) \rightarrow (\Omega_1, U_1)$.

\begin{enumerate}
\item[(ET1)] {\em (Cutting a base (see Fig. \ref{ET1}))}.
Let $\lambda$ be a base in $\Omega$ and $p$ an internal boundary of $\lambda$ (that is, $p \neq
\alpha(\lambda), \beta(\lambda)$) with a boundary connection $(p, \lambda, q)$. Then we cut the
base $\lambda$ at $p$ into two new bases $\lambda_1$ and $\lambda_2$, and cut $\overline\lambda$ at
$q$ into the bases $\overline\lambda_1$ and $\overline\lambda_2$.

\begin{figure}[htbp]
\centering{\mbox{\psfig{figure=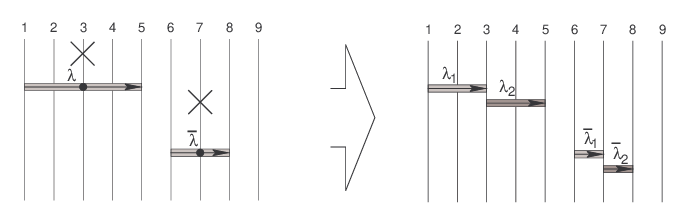}}}
\caption{Elementary transformation (ET1).}
\label{ET1}
\end{figure}

\item[(ET2)] {\em (Transfering a base (see Fig. \ref{ET2}))}.
If a base $\lambda$ of $\Omega$ contains a base $\mu$ (that is, $\alpha(\lambda) \leqslant \alpha(\mu) <
\beta(\mu) \leqslant \beta(\lambda)$) and all boundaries on $\mu$ are $\lambda$-tied by boundary some
connections then we transfer $\mu$ from its location on the base $\lambda$ to the corresponding
location on the base $\overline\lambda$.

\begin{figure}[htbp]
\centering{\mbox{\psfig{figure=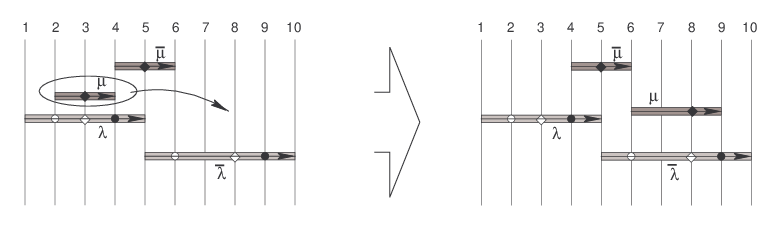}}}
\caption{Elementary transformation (ET2).}
\label{ET2}
\end{figure}

\item[(ET3)] {\em (Removal of a pair of  matched bases (see Fig. \ref{ET3}))}.
If the bases $\lambda$ and $\overline\lambda$ are {\em  matched} (that is, $\alpha(\lambda) =
\alpha(\overline\lambda), \beta(\lambda) = \beta(\overline\lambda)$) then we remove $\lambda,
\overline\lambda$ from $\Omega$.

\begin{figure}[htbp]
\centering{\mbox{\psfig{figure=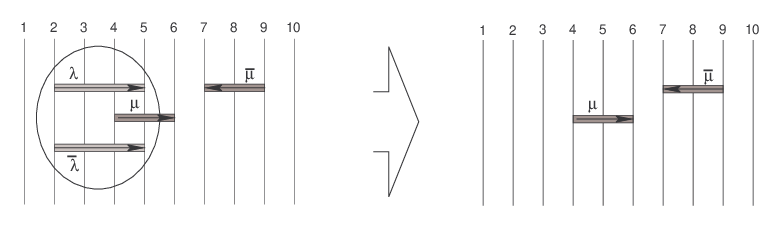}}}
\caption{Elementary transformation (ET3).}
\label{ET3}
\end{figure}

\begin{remark}
Observe, that $\Omega$ and $\Omega_1$, where $\Omega_1 = ETi(\Omega)$ for $i \in \{1,2,3\}$ have the
same set of variables $h$ and the bijection $h_i \rightarrow h_i,\ i \in [1,\rho]$ induces an
isomorphism $G_\Omega \rightarrow G_{\Omega_1}$. Moreover, $U$ is a solution of $\Omega$ if and only
if $U$ is a solution of $\Omega_1$.
\end{remark}

\item[(ET4)] {\em (Removal of a lone base (see Fig. \ref{ET4}))}.
Suppose, a base $\lambda$ in $\Omega$ does not {\em intersect} any other base, that is, the items
$h_{\alpha(\lambda)}, \ldots, h_{\beta(\lambda) - 1}$ are contained only inside of the base
$\lambda$.

\begin{figure}[htbp]
\centering{\mbox{\psfig{figure=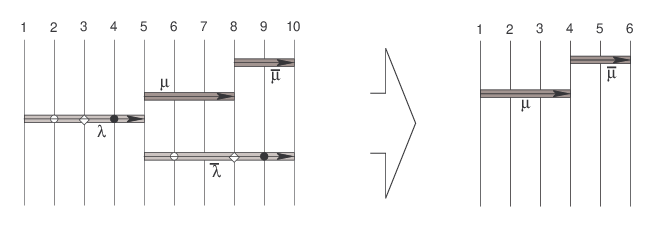}}}
\caption{Elementary transformation (ET4).}
\label{ET4}
\end{figure}

Suppose also that all boundaries in $\lambda$ are $\lambda$-tied, that is, for every $i$
($\alpha(\lambda) < i \leqslant \beta(\lambda) - 1$) there exists a boundary $b(i)$ such that $(i,
\lambda, b(i))$ is a boundary connection in $\Omega$. Then we remove the pair of bases $\lambda$ and
$\overline\lambda$ together with all the boundaries $\alpha(\lambda) + 1, \ldots, \beta(\lambda) - 1$
(and rename the rest $\beta(\lambda) - \alpha(\lambda) - 1$ of the boundaries correspondingly).

We define the isomorphism $\pi: G_{\Omega} \rightarrow G_{\Omega_1}$ as follows:
$$\pi(h_j) = h_j\ {\rm if}\ j < \alpha(\lambda)\ {\rm or}\ j \geqslant \beta(\lambda)$$
\[\pi(h_i) = \left\{\begin{array}{ll}
h_{b(i)} \cdots h_{b(i) - 1}, & if\ \varepsilon(\lambda) = \varepsilon(\overline\lambda),\\
h_{b(i)} \cdots h_{b(i-1) - 1},& if\ \varepsilon(\lambda) = -\varepsilon(\overline\lambda)
\end{array}
\right.
\]
for $\alpha + 1 \leqslant i \leqslant \beta(\lambda) - 1$.

\begin{figure}[htbp]
\centering{\mbox{\psfig{figure=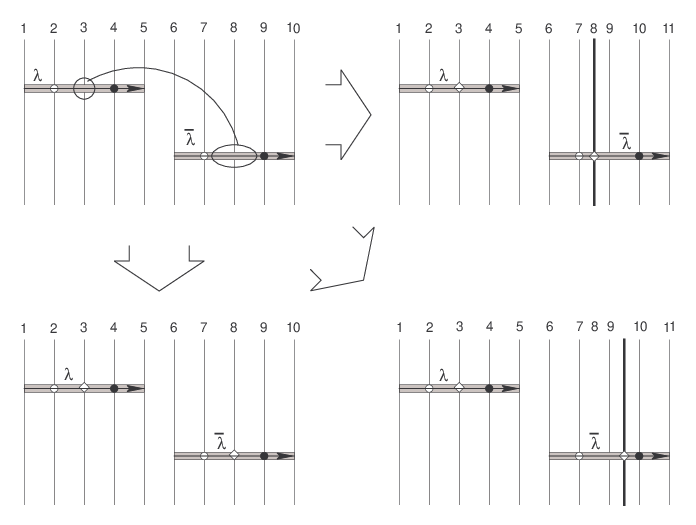}}}
\caption{Elementary transformation (ET5).}
\label{ET5}
\end{figure}

\item[(ET5)] {\em (Introduction of a boundary (see Fig. \ref{ET5}))}. Suppose a point $p$ in a base
$\lambda$ is not $\lambda$-tied. The transformation (ET5) $\lambda$-ties it. To this end, denote by
$u_\lambda$ the element of $CDR(\Lambda ,Z)$ corresponding to $\lambda$ and let $u'_\lambda$ be the
beginning of this word ending at $p$. Then we perform one of the following two transformations
according to where the end of $u'_\lambda$ on $\overline\lambda$ is situated:

\begin{enumerate}
\item If the end of $u'_\lambda$ on $\overline\lambda$ is situated on the boundary $q$ then we
introduce the boundary connection $(p, \lambda ,q)$. In this case the corresponding isomorphism
$\pi : G_\Omega \rightarrow G_{\Omega_1}$ is induced by the bijection $h_i \rightarrow h_i,\ i \in
[1,\rho]$. (If we began with the group $\tilde G$ with non-regular length function as in Remark
\ref{nonreg} this is the only place where  $\pi : G_\Omega \rightarrow G_{\Omega_1}$ may be a proper
epimorphism, but its restriction on $\tilde G$ is still an isomorphism.)

\item If the end of $u'_\lambda$ on $\overline\lambda$ is situated between $q$ and $q+1$ then we
introduce a new boundary $q'$ between $q$ and $q+1$ (and rename all the boundaries), and also
introduce a new boundary connection $(p, \lambda, q')$. In this case the corresponding isomorphism
$\pi: G_\Omega \rightarrow G_{\Omega_1}$ is induced by the map $\pi(h) = h$, if $h \neq h_q$, and
$\pi(h_q) = h_{q'} h_{q'+1}$.
\end{enumerate}
\end{enumerate}

\subsection{Derived transformations and auxiliary transformations}
\label{se:5.2half}

In this section we define complexity of a generalized equation and describe several useful
``derived'' transformations of generalized equations. Some of them can be realized as finite
sequences of elementary transformations, others result in equivalent generalized equations but
cannot be realized by finite sequences of elementary moves.

A boundary is {\em open} if it is an internal boundary of some base, otherwise it is {\em closed}.
A section $\sigma = [i, \ldots, i + k]$ is said to be {\em closed} if the boundaries $i$ and $i+k$
are closed and all the boundaries between them are open.

Sometimes it will be convenient to subdivide all sections of $\Omega$ into {\em active} (denoted
$A\Sigma_\Omega$) and {\em non-active} sections. For an item $h$ denote by $\gamma(h)$ the number of
bases containing $h$. An item $h$ is called {\em free} is it meets no base, that is, if $\gamma(h) =
0$. Free variables are transported to the very end of the interval behind all items in $\Omega$ and
they become non-active.

\begin{enumerate}
\item[(D1)] {\em (Closing a section)}.
Let $\sigma$ be a section of $\Omega$. The transformation (D1) makes the section $\sigma$ closed.
Namely, (D1) cuts all bases in $\Omega$ through the end-points of $\sigma$.

\item[(D2)] {\em (Transporting a closed section)}.
Let $\sigma$ be a closed section of a generalized equation $\Omega$. We cut $\sigma$ out of the
interval $[1,\rho_\Omega]$ together with all the bases on $\sigma$ and put $\sigma$ at the end of the
interval or between any two consecutive closed sections of $\Omega$. After that we correspondingly
re-enumerate all the items and boundaries of the latter equation to bring it to the proper form.
Clearly, the original equation $\Omega$ and the new one $\Omega'$ have the same solution sets and
their coordinate groups are isomorphic

\item[(D3)] {\em (Moving free variables to the right)}.
Suppose that $\Omega$ contains a free variable $h_q$ in an active section. Here we close the section
$[q, q + 1]$ using (D1), transport it to the very end of the interval behind all items in $\Omega$
using (D2). In the resulting generalized equation $\Omega'$ the transported section becomes a
non-active section.

\item[(D4)] {\em (Deleting a complete base)}.
A base $\mu$ of $\Omega$ is called {\em complete} if there exists a closed section $\sigma$ in
$\Omega$ such that $\sigma = [\alpha(\mu), \beta(\mu)]$.

Suppose $\mu$ is an active  complete base of $\Omega$ and $\sigma$ is a closed section such that
$\sigma = [\alpha(\mu),\beta(\mu)]$. In this case using (ET5), we transfer all bases from $\mu$ to
$\overline\mu$, then using (ET4) we remove the lone base $\mu$ together with the section
$\sigma$.

\item[(D5)] {\em (Finding the kernel)}.
We will give a definition of eliminable base for an equation $\Omega$ that does not have any boundary
connections. An active base $\mu \in A\Sigma_\Omega$ is called {\em eliminable} if at least one of
the following holds

\begin{enumerate}
\item[(a)] $\mu$ contains an item $h_i$ with $\gamma(h_i) = 1$,

\item[(b)] at least one of the boundaries $\alpha(\mu), \beta(\mu)$ is different from $1, \rho +1$
and does not touch any other base (except for $\mu$).
\end{enumerate}

The process of finding the kernel works as follows. We cut the bases of $\Omega$ along all the
boundary connections thus obtaining the equation without boundary connections, then
consequently remove eliminable bases until no eliminable base is left in the equation. The
resulting generalized equation is called the {\em kernel} of $\Omega$ and we denote
it by $Ker(\Omega)$. It is easy to see that $Ker(\Omega)$ does not depend on a particular removal
process. Indeed, if $\Omega$ has two different eliminable bases $\mu_1,\ \mu_2$, and deletion of
$\mu_i$ produces an equation $\Omega_i$ then by induction (on the number of eliminations)
$Ker(\Omega_i)$ is uniquely defined for $i = 1, 2$. Obviously, $\mu_1$ is still eliminable in
$\Omega_2$, as well as $\mu_2$ is eliminable in $\Omega_1$. Now eliminating $\mu_1$ and $\mu_2$ from
$\Omega_2$ and $\Omega_1$ we get one and the same equation $\Omega_0$. By induction, $Ker(\Omega_1)
= Ker(\Omega_0) = Ker(\Omega_2)$ hence the result. We say that an item $h_i$ {\it belongs to
the kernel} ($h_i \in Ker(\Omega)$), if $h_i$ belongs to at least one base in the kernel.
Notice that the kernel can be empty.

Also, for an equation $\Omega$ by $\overline{\Omega}$ we denote the equation which is obtained from
$\Omega$ by deleting all free variables. Obviously,
$$G_\Omega = G_{\overline{\Omega}} \ast F(\overline Y),$$
where $\overline Y$ is the set of free variables in $\Omega$.

Let us consider what happens in this process on the group level.

We start with the case when just one base is eliminated. Let $\mu$ be an eliminable base in $\Omega
= \Omega(h_1, \ldots, h_\rho)$. Denote by $\Omega_1$ the equation resulting from $\Omega$ by
eliminating $\mu$.

\begin{enumerate}
\item[(a)] Suppose $h_i \in \mu$ and $\gamma(h_i) = 1$. Then the variable $h_i$ occurs only once in
$\Omega$ - precisely in the equation $s_\mu = 1$ corresponding to the base $\mu$. Therefore, in the
group $G_\Omega$ the relation $s_\mu = 1$ can be written as $h_i = w$, where $w$ does not contain
$h_i$. Using Tietze transformations we can rewrite the presentation of $G_\Omega$ as $G_{\Omega'}$,
where $\Omega'$ is obtained from $\Omega$ by deleting $s_\mu$ and the item $h_i$. It follows
immediately that
$$G_{\Omega_1} \simeq G_{\Omega'} \ast \langle h_i \rangle$$
and
\begin{equation}
\label{eq:ker1}
G_\Omega \simeq G_{\Omega'} \simeq G_{\Omega_1} \ast F(B)
\end{equation}
for some free group $F(B)$.

\item[(b)] Suppose now that $\mu$ satisfies  case (b) above with respect to a boundary $i$. Then in
the equation $s_\mu = 1$ the variable $h_{i-1}$ either occurs only once, or it occurs precisely twice
and in this event the second occurrence of $h_{i-1}$ (in $\overline\mu$) is a part of the subword
$(h_{i-1} h_i)^{\pm 1}$. In both cases it is easy to see that the tuple
$$(h_1, \ldots, h_{i-2}, s_\mu, h_{i-1} h_i, h_{i+1}, \ldots, h_\rho)$$
generates $G$. Therefore, by eliminating the relation $s_\mu = 1$ we can rewrite the presentation of
$G_\Omega$ in the generators $\overline Y = (h_1, \ldots, h_{i-2}, h_{i-1}h_i, h_{i+1}, \ldots,
h_\rho)$. Observe also that any other equation $s_\lambda = 1$ ($\lambda \neq \mu$) of $\Omega$ either
does not contain the variables $h_{i-1}, h_i$, or it contains them as parts of the subword
$(h_{i-1}h_i)^{\pm 1}$, that is, any such a word $s_\lambda$ can be expressed as a word
$w_\lambda(\overline Y)$ in terms of the generators $\overline Y$. This shows that
$$G_\Omega \simeq \langle \overline Y \mid \{w_\lambda(\overline Y) \mid \lambda \neq \mu\} \rangle
\simeq G_{\Omega'},$$
where $\Omega'$ is a generalized equation obtained from $\Omega_1$ by deleting the boundary $i$.
Denote by $\Omega''$ an equation obtained from $\Omega'$ by adding a free variable $z$ to the right
end of $\Omega'$. It follows now that
$$G_{\Omega_1} \simeq  G_{\Omega''} \simeq G_\Omega \ast \langle z \rangle$$
and
\begin{equation}
\label{eq:ker2}
G_\Omega \simeq G_{\overline{\Omega'}} \ast F(K)
\end{equation}
for some free group $F(K)$.
\end{enumerate}

By induction on the number of steps in the process we obtain the following lemma.

\begin{lemma}
\label{7-10}
If $\Omega $ is a generalized equation, then
$$F_\Omega \simeq F_{\overline{Ker(\Omega)}} \ast F(K)$$
where $F(K)$ is a free group on $K$.
\end{lemma}
\begin{proof} Let
$$\Omega = \Omega_0 \rightarrow \Omega_1 \rightarrow \cdots \rightarrow \Omega_l = Ker(\Omega)$$
be a linear elimination process. It is easy to see (by induction on $l$)
that for every $j = 0, \ldots, l-1$
$$\overline{Ker(\Omega_j)} = \overline{Ker(\overline{\Omega_j})}.$$
Moreover, if $\Omega_{j+1}$ is obtained from $\Omega_j$ as in the case (b) above, then (in the
notations above)
$$\overline{Ker(\Omega_j)_1} = \overline{Ker(\Omega_j')}.$$
Now the statement of the lemma follows from the remarks above and equalities (\ref{eq:ker1}) and
(\ref{eq:ker2}).
\end{proof}

\item[(D6)] {\em (Linear elimination)}.
Suppose that in $\Omega$ there is $h_i$ in an active section with $\gamma(h_i) = 1$ and such that
$|h_i|$ is comparable with the length of the active section. In this case we say that $\Omega$ is
{\em linear in $h_i$}.

If $\Omega$ is linear in $h_i$ in an active section such that both boundaries $i$ and $i+1$ are
closed then we remove the closed section $[i,i+1]$ together with the lone base using (ET4).

If there is no such $h_i$ but $\Omega$ is linear in some $h_i$ in an active section such that one
of the boundaries $i,i+1$ is open, say $i+1$, and the other is closed, then we perform (ET5) and
$\mu$-tie $i+1$ through the only base $\mu$ it intersects. Next, using (ET1) we cut $\mu$ in $i+1$
and then we delete the closed section $[i,i+1]$ by (ET4).

Suppose there is no $h_i$ as above but $\Omega$ is linear in some $h_i$ in an active section such
that both boundaries $i$ and $i+1$ are open. In addition, assume that there is a closed section
$\sigma$ containing exactly two (not matched) bases $\mu_1$ and $\mu_2$, such that $\sigma =
\sigma(\mu_1) = \sigma(\mu_2)$ and in the generalized equation $\widetilde{\Omega}$ (see the derived
transformation (D3)) all the bases obtained from $\mu_1,\mu_2$ by (ET1) in constructing
$\widetilde{\Omega}$ from $\Omega$, do not belong to the kernel of $\widetilde{\Omega}$. Here, using
(ET5), we $\mu_1$-tie all the boundaries inside of $\mu_1$, then using (ET2) we transfer $\mu_2$ onto
$\overline\mu_1$, and remove $\mu_1$ together with the closed section $\sigma$ using (ET4).

Suppose now that $\Omega$  satisfies the first assumption of the previous paragraph and does not
satisfy the second one. In this event we close the section $[i, i+1]$ using (D1) and remove it
using (ET4).

\begin{lemma}\label{lin-el}
Suppose that the process of linear elimination continues infinitely and there is a corresponding
sequence of generalized equations
$$\Omega \rightarrow \Omega_1 \rightarrow \cdots \rightarrow \Omega_k \rightarrow \cdots.$$
Then
\begin{enumerate}
\item (\cite{KMIrc}, Lemma 15) The number of different generalized equations that appear in the
process is finite. Therefore some generalized equation appears in this process infinitely many
times.

\item (\cite{KMIrc}, Lemma 15) If $\Omega_j = \Omega_k,\ j < k$ then $\pi(j,k)$ is an isomorphism,
invariant with respect to the kernel, namely $\pi(j,k)(h_i) = h_i$ for any variable $h_i$ that
belongs to some base in $Ker(\Omega)$.

\item The interval for the equation $\Omega_j$ can be divided into two disjoint parts, each being
the union of closed sections, such that one part is a generalized equation $Ker(\Omega)$ and the
other part is non-empty and corresponds to a generalized equation $\Omega'$, such that $G_{\Omega'}
= F(K)$ is a free group on variables $K$ and $G_\Omega = G_{Ker(\Omega)} \ast F(K)$.
\end{enumerate}
\end{lemma}
\begin{proof} We will only prove the third statement. In the process of linear elimination we do not
remove or cut variables $h_i$ that belong to the kernel. Since we do not cut variables that belong to
the kernel, we can remove parts of bases that cover some of these variables only finitely many times.
Therefore, bases that are cut in the process belong to the part of the interval that is disjoint from
the $Ker(\Omega)$. Together with Lemma \ref{7-10} this proves the third statement.
\end{proof}

\item[(D7)] {\em (Tietze cleaning)}.
This transformation consists of four transformations performed consecutively
\begin{enumerate}
\item[(a)] linear elimination: if the process of linear elimination goes infinitely we replace the
equation by its kernel,
\item[(b)] deleting all pairs of matched bases,
\item[(c)] deleting all complete bases,
\item[(d)] moving all free variables to the right.
\end{enumerate}

\item[(D8)] {\em (Entire transformation)}.
We need a few definitions. A base $\mu$ of the equation $\Omega$ is called a {\em leading} base if
$\alpha(\mu) = 1$. A leading base is said to be {\it maximal} (or a {\em carrier base}) if
$\beta(\lambda) \leqslant \beta(\mu)$ for any other leading base $\lambda$. Let $\mu$ be a carrier base
of $\Omega$. Any active base $\lambda \neq \mu$ with $\beta(\lambda) \leqslant \beta(\mu)$ is called a
{\em transfer} base (with respect to $\mu$).

Suppose now that $\Omega$ is a generalized equation with $\gamma(h_i) \geqslant 2$ for each $h_i$ in
the active part of $\Omega$ and such that $|h_i|$ is comparable with the length of the active part.
{\em Entire transformation} is a sequence of elementary transformations which are performed as
follows. We fix a carrier base $\mu$ of $\Omega$. We transfer all transfer bases from $\mu$ onto
$\overline \mu$. Now, there exists some $i < \beta(\mu)$ such that $h_1, \ldots, h_i$ belong to only
one base $\mu$, while $h_{i+1}$ belongs to at least two bases. Applying (ET1) we cut $\mu$ along the
boundary $i+1$. Finally, applying (ET4) we delete the section $[1,i+1]$.
\end{enumerate}

\subsection{Complexity of a generalized equation and Delzant-Potyagailo complexity $c(G)$ of a group $G$}
\label{subs:complexity}

Denote by $\rho_A$ the number of variables $h_i$ in all active sections of $\Omega$, by $n_A =
n_A(\Omega)$ the number of bases in all active sections of $\Omega$, by $\nu'$ the number of open
boundaries in the active sections, and by $\sigma'$ the number of closed boundaries in the active
sections.

For a closed section $\sigma \in \Sigma_\Omega$ denote by $n(\sigma), \rho(\sigma)$ the number of
bases and, respectively, variables in $\sigma$.
$$\rho _A = \rho_A(\Omega) = \sum_{\sigma \in A\Sigma_\Omega} \rho(\sigma),$$
$$n_A = n_A(\Omega) = \sum_{\sigma \in A\Sigma_\Omega} n(\sigma).$$

The {\em complexity} of the  generalized  equation $\Omega $ is the number
$$\tau = \tau (\Omega) = \sum_{\sigma \in A\Sigma_\Omega} \max\{0, n(\sigma) - 2\}.$$

Notice that the entire transformation (D5) as well as the cleaning process (D4) do not increase
complexity of equations.

Below we recall Delzant-Potyagailo's result (see \cite{DP}). A family ${\cal C}$ of subgroups of a
torsion-free group $G$ is called {\em elementary} if
\begin{enumerate}
\item[(a)] ${\cal C}$ is closed under taking subgroups and conjugation,
\item[(b)] every $C \in {\cal C}$ is contained in a maximal subgroup $\overline{C} \in {\cal C}$,
\item[(c)] every $C \in {\cal C}$ is small (does not contain $F_2$ as a subgroup),
\item[(d)] all maximal subgroups from ${\cal C}$ are malnormal.
\end{enumerate}
$G$ admits a {\em hierarchy} over ${\cal C}$ if the process of decomposing $G$ into an amalgamated
product or an HNN-extension over a subgroup from ${\cal C}$, then decomposing factors of $G$ into
amalgamated products and/or HNN-extensions over a subgroup from ${\cal C}$ etc. eventually stops.

\begin{proposition} (\cite{DP})
If $G$ is a finitely presented group without $2$-torsion and ${\cal C}$ is a family of elementary
subgroups of $G$ then $G$ admits a hierarchy over ${\cal C}$.
\end{proposition}

\begin{cor}
If $G$ is a finitely presented $\Lambda$-free group then $G$ admits a hierarchy over the family of
all abelian subgroups.
\end{cor}

There is a notion of complexity of a group $G$ defined in \cite{DP} and denoted by $c(G)$. We will
only use the following statement that follows from there.

\begin{proposition} (\cite{DP})
\label{DP}
If $G$ is a finitely presented $\Lambda$-free group. Let $\Gamma$ be an abelian decomposition of $G$
as a fundamental group of a graph of groups with more than one vertex and with each edge group being
maximal abelian at least in one of its vertex groups. Then for each vertex group $G_v$, $c(G_v) <
c(G)$.
\end{proposition}

\subsection{Rewriting process for $\Omega$}
\label{se:5.2}

In this section we describe a rewriting process (elimination process) for a generalized equation
$\Omega$ corresponding to $G$. Performing the elimination process we eventually detect a decomposition
of $G$ as a free product or (if it is freely indecomposable) as the fundamental group of a graph of
groups with vertex groups of three types: QH vertex groups,  abelian vertex groups (corresponding to
periodic structures, see below), non-QH, non-abelian vertex groups (we will call them {\em weakly
rigid} meaning that we do not split them in this particular decomposition). We also can detect
splitting of $G$ as an HNN-extension with stable letter infinitely longer than generators of the
abelian associated subgroups. After obtaining such a decomposition we continue the elimination
process with the generalized equation corresponding to free factors of $G$  or to weakly rigid
subgroups of $G$ (we will show that this generalized equation can be naturally obtained from the
generalized equation $\Omega$. The Delzant-Potyagailo complexity of factors in a free decomposition
and complexity of weakly rigid subgroups is smaller than the complexity of $G$ (denoted $c(G)$)
\cite{DP}. In the case of an HNN extension we will show that the complexity $\tau$ of the generalized
equation corresponding to a weakly rigid subgroup is smaller that the complexity of $\Omega$.

We assume that $\Omega$ is in standard form, namely, that transformations (ET3), (D3) and (D4) have
been applied to $\Omega$ and that on each step we apply them to the generalized equation before
applying any other transformation.

Let $\Omega $ be a generalized equation. We construct a path $T(\Omega)$ (with associated structures),
as a directed path oriented from the root $v_0$, starting at $v_0$ and proceeding by induction on the
distance $n$ from the root.

We start with a general description of the path $T(\Omega)$. For each vertex $v$ in $T(\Omega)$ there
exists a unique generalized equation $\Omega_v$ associated with $v$. The initial equation $\Omega$ is
associated with the root $v_0$, $\Omega_{v_0} = \Omega$. In addition there is a homogeneous system of
linear equations $\Sigma_v$ with integer coefficients on the lengths of  variables of $\Omega _v$.
 We take $\Sigma_{v_0}$ to be empty. For each edge $v \to v'$ (here $v$ and $v'$ are the
origin and the terminus of the edge) there exists an isomorphism $\pi(v,v') : G_{\Omega_v} \to
G_{\Omega_v'}$ associated with $v \to v'$.

If
$$v_0 \rightarrow v_1 \rightarrow \cdots \rightarrow v_s \rightarrow u$$
is a dubpath of  $T(\Omega)$, then by $\pi(v,u)$ we denote composition of corresponding isomorphisms
$$\pi(v,u) = \pi(v,v_1) \circ \cdots \circ \pi(v_s,u).$$

If $v \to v'$ is an edge then there exists a finite sequence of elementary or derived transformations
from $\Omega_v$ to $\Omega_{v'}$ and the isomorphism $\pi(v,v')$ is a composition of the
isomorphisms corresponding to these transformations. We also assume that active (and non-active)
sections in $\Omega_{v'}$ are naturally inherited from $\Omega_v$, if not said otherwise. Recall that
initially all sections are active.

Suppose the path $T(\Omega)$ is constructed by induction up to a level $n$ and suppose $v$ is a
vertex at distance $n$ from the root $v_0$. We describe now how to extend the path from $v$. The
construction of the outgoing edge at $v$ depends on which case described below takes place at the
vertex $v$. There are three possible cases.

\begin{itemize}
\item {\bf Linear case:} there exists $h_i$ in the active part such that $|h_i|$ is comparable with
the length of the active part and $\gamma(h_i) = 1$.

\item {\bf Quadratic and almost quadratic case:} $\gamma(h_i) = 2$ for all $h_i$ in the active part
such that $|h_i|$ is comparable with the length of the active part.

\item {\bf General JSJ case:} $\gamma(h_i) \geqslant 2$ for all $h_i$ in the active part such that
$|h_i|$ is comparable with the length of the active part, and there exists such $h_i$ that $\gamma(h_i)
> 2$.
\end{itemize}

\subsubsection{Linear case}
\label{linear_case}

We apply Tietze cleaning (D7) at the vertex $v_n$ if it is possible. We re-write the system of linear
equations $\Sigma_{v_n}$  in new variables and obtain a new system $\Sigma_{v_{n+1}}$.

If $\Omega_{v_{n+1}}$ splits into two parts, $\Omega_{v_{n+1}}^{(1)} = Ker(\Omega_{v_n})$ and
$\Omega_{v_{n+1}}^{(2)}$ that corresponds to a free group $F(K)$, then we put the free group section
$\Omega_{v_{n+1}}^{(2)}$ into a non-active part and thus decrease the complexity. It may happen that
the kernel is empty, then the process terminates.

If it is impossible to apply Tietze cleaning (that is $\gamma(h_i) \geqslant 2$ for any $h_i$ in the
active part of $\Omega_v$ comparable to the length of the active part), we apply the entire
transformation.

{\bf Termination condition:} $\Omega_v$ does not contain active sections. In this case the vertex
$v$ is called a {\it leaf} or an {\it end vertex}.

\subsubsection{Quadratic case}
\label{quad_case}

Suppose $\Omega_v$ satisfies the condition $\gamma_i \leqslant 2$ for each $h_i$ in the active part
and $\gamma_i = 2$ for each $h_i$ in the active part comparable with the length of the active part.
First of all, we fill in all the $h_i's$ in the active part such that $\gamma_i = 1$ by new (infinitely
short) bases $\mu$ such that $\overline\mu$ covers a new variable that we add to the non-active part.

We apply the entire transformation (D8), then apply Tietze cleaning (D7), if possible, then again
apply entire transformation, etc. In this process we, maybe, will remove some pairs of matching bases
decreasing the complexity. Eventually we either end up with empty active part or the process will
continue infinitely, and the number of bases in the active part will be constant. Let $\alpha _0 +1$
be the leftmost  boundary that belongs to a base not participating in the process.

Let us analyze the structure of $G_{\Omega_v}$ in this case. Let $\sigma$ be a section in the active
part $\sigma = [h_1, \ldots, h_{\alpha _0}]$. If $\sigma$ is not closed, then cutting the second base
at $\alpha _0$ we make $\sigma$ a closed section. Let $F_1$ be a free group with basis $\{h_1, \ldots,
h_{\alpha _0}\}$. If $\sigma$ contains an open boundary $j$ then we can consider a new generalized
equation $\Omega_v'$ obtained from $\Omega_v$ by replacing the product $h_{j-1} h_j$ by a new variable
$h_{j-1}'$ and represent $G_{\Omega_v} = G_{\Omega_v'} \ast \langle h_j \rangle$. Therefore we can
suppose that $\sigma$ does not contain open boundaries. We say that $\mu$ is a {\em quadratic} base
if $\sigma$ contains $\mu$ and $\overline\mu$, otherwise $\mu$ is a {\em quadratic-coefficient base}.
Denote the set of quadratic-coefficient bases by $C$. Suppose that $\sigma$ contains quadratic bases.
Let $F_1 / \sigma$ be the quotient of $F_1$ over the normal closure of elements $h[\alpha(\mu),
\beta(\mu)] h[\alpha(\overline\mu),\beta(\overline\mu)]^{-1}$. Let ${\mathcal M}$ be a set that
contains exactly one representative of each pair of double bases on $\sigma$ and contains also each
base $\mu$ such that $\mu \in \sigma$ and $\bar\mu \notin \sigma$. If we identify each base on
$\sigma$ with its double then the product $h_1 \cdots h_i$ can be written as a product of bases from
${\mathcal M}$ in exactly two different ways: $\mu_{i_1} \cdots \mu_{i_k}$ and $\mu_{j_1} \cdots
\mu_{j_t}$. Then $F_1 / \sigma$ is isomorphic to the quotient of the free group $F(\mathcal M)$ over
the relation $\mu_{i_1} \cdots \mu_{i_k} = \mu_{j_1} \cdots \mu_{j_t}$. Every element of $\mathcal
M$ occurs in this relation at most twice. Applying an automorphism of $F(\mathcal M)$ identical on
$C$, we can obtain another basis $X \cup T \cup C$ of this group such that in this basis the relation
has form of a standard quadratic equation in variables from $X$ with coefficients in $F(C)$ (see
\cite{LS}, Section 1.7). Note that the variables from $T$ do not participate in any relations. The
quadratic equation corresponding to $\sigma$ can be written in the standard form with coefficients
expressed in terms of non-active variables.

Here we have to give some definitions.

\begin{definition}
A standard quadratic equation over the group $G$ is an equation of the one of the following forms
(below $d,c_i$ are nontrivial elements from $G$):
\begin{equation}
\label{eq:st2}
\prod_{i=1}^n [x_i,y_i] \prod_{i=1}^m z_i^{-1} c_i z_i d = 1,\ \ \ n, m \geqslant 0, m + n \geqslant 1;
\end{equation}
\begin{equation}
\label{eq:st4}
\prod_{i=1}^n x_i^2 \prod_{i=1}^m z_i^{-1} c_i z_i d = 1, \ \ \ n, m \geqslant 0, n + m \geqslant 1,
\end{equation}
where $c_1, \ldots, c_m, d$ are coefficients.
\end{definition}

Let $W$ be a strictly quadratic word over a group $G$. Then there is a $G$-automorphism $f \in
Aut_G(G[X])$ such that  $W^f$ is a standard quadratic word over $G.$

To each quadratic equation one can associate a punctured surface. For example, the orientable surface
associated to equation \ref{eq:st2} will have genus $n$ and $m + 1$ punctures.

Put
$$\kappa(S) = |X| + 1,$$

\begin{definition}
\label{regular}
Let $S = 1$ be a standard quadratic equation over a group $G$ which has a solution in $G$. The equation
$S(X) = 1$ is {\em regular} if $\kappa(S) \geqslant 4$ (equivalently, the Euler characteristic of the
corresponding punctured surface is at most $-2$) and there is a non-commutative solution of $S(X) = 1$
in $G$, or it is an equation of the type $[x,y] d = 1$, or $[x_1,y_1][x_2,y_2] = 1$.
\end{definition}

If the quadratic equation corresponding to $\sigma$ is regular, then $G_{\Omega_v}$ has a presentation
as a fundamental group of a graph of groups with QH-subgroup corresponding to this equation. (For
example, the QH-subgroup corresponding to the standard equation (\ref{eq:st2}) has the presentation
$$\langle x_1, y_1, \ldots, x_n, y_n, p_1, \ldots, p_{m+1} \mid \prod_{i=1}^n [x_i,y_i]
\prod_{i=1}^{m+1} p_i = 1 \rangle,$$
and the edge groups are $\langle p_i \rangle, i = 1, \ldots, m+1$.)

If the equation is not regular and there are quadratic bases, then $G_{\Omega_v}$ splits as an
HNN-extension over an abelian subgroup. If there are infinitely small $h_i$'s in the quadratic part
such that $\gamma(h_i) = 1$ then $G_{\Omega_v'}$ splits as a free product of the group of the
generalized equation on the non-quadratic part (obtained by removing the quadratic part and all bases
that have doubles in quadratic part) and a free group.

We say that a base {\em participates} in the infinite process of entire transformation if it
participates infinitely many times as a leading or transfer base.

\begin{lemma}
If $\sigma$ has quadratic-coefficient bases, and the entire transformation goes infinitely, then
after a finite number of steps there will be quadratic bases that have length infinitely larger than
all participating quadratic coefficient bases.
\end{lemma}
\begin{proof}
After finite number of steps quadratic-coefficient bases can only be used as transfer bases (otherwise
quadratic bases would be leaving $\sigma$, and this cannot happen infinitely many times). Every time
such a base of length $\ell$ used as a transfer base, the length of the quadratic part decreases by
$\ell$. For the process to go infinitely, $\ell$ must be infinitely less than the length of the
quadratic part. Therefore there is a quadratic base that have length infinitely larger than all
participating quadratic coefficient bases.
\end{proof}

If $\sigma$ does not have quadratic-coefficient bases then $G_{R(\Omega_v)}$ splits as a free product
with one factor being a closed surface group or a free group.

We move $\sigma$ into a non-active part and thus decrease the complexity $\tau$. Eventually in this
case we will make the active part empty.

\subsubsection{Almost Quadratic case}
\label{almost_quad_case}

In this case $\gamma_i = 2$ for all $h_i$ in the active part such that $|h_i|$ is comparable with
the length of the active part (some short bases may be covered once or more than twice).
This case is, actually, a part of the case considered in the next
subsection, but we consider it separately to make the subdivision into the cases depend only on
$\gamma_i$ for the variables $h_i$ of the maximal height.

We apply the transformation (D1) to close  quadratic sections that begin with  bases of the maximal
height and put them in front of the interval. This may increase the complexity by two times the number
of closed sections that we put in front of the interval applying (D1). Then we apply the entire
transformation. Every time when the content of one of the closed sections in front of the interval is
moved to the other closed section, the complexity decreases by two. If the process  goes infinitely
then eventually the number of quadratic bases of maximal height does not decrease. After a finite
number of steps participating quadratic-coefficient bases can only be used as transfer bases. And,
therefore, there will be participating quadratic bases that have length infinitely larger than all
participating quadratic coefficient bases.

We fill in all the $h_i's$ in the active part such that $\gamma_i = 1$ by new (infinitely short)
bases $\mu$ such that $\overline\mu$ covers a new free variable in the non-active part appearing only
once. Then we work with the quadratic part the same way as in  the quadratic case. We repeat the
described transformation until there is no quadratic base on the active part that has length comparable
with the length of the remaining active part. Then we  consider the remaining generalized equation
in the active part. We remove from the active part doubles of all quadratic coefficient bases that
belong to non-active part (doing this we may create new boundaries). We will, certainly remember the
relations corresponding to these pairs of bases. In this case the remaining generalized equation has
smaller complexity. Relations corresponding to the quadratic sections that we made non-active show
that $G_{\Omega}$ is an HNN-extension of the subgroup generated by the variables in the active part
and (maybe) a free group.

Removing these double bases we have to add equations to $\Sigma_{v_{i+1}}$ that guarantee that the
associated cyclic subgroups are generated by elements of the same length.

\subsubsection{General JSJ-case}
\label{gen_jsj_case}

Generalized equation $\Omega_v$ satisfies the condition $\gamma_i \geqslant 2$ for each $h_i$ in the
active part such that $|h_i|$ is comparable with the length of the active part, and $\gamma_i > 2$
for at least one $h_i$. First of all, we fill in all the $h_i's$ in the active part such that
$\gamma_i = 1$ by new (infinitely short) bases with doubles corresponding to free variables in the
non-active part. We apply the transformation (D1) to close the quadratic part and put it in front of
the interval.

\paragraph{(a) QH-subgroup case.}

Suppose that the entire transformation of the quadratic part (D8) goes infinitely. Then the quadratic
part of $\Omega_v$ (or the initial  section from the beginning of the  quadratic part until the first
base on the quadratic part that does not participate in the entire transformation) corresponds to a
QH-vertex or to the representation of $G_\Omega$ as an HNN-extension, and there is a quadratic base
(on this section) that is infinitely longer than all the quadratic coefficient bases (on this section).
We work with the quadratic part the same way as in the quadratic case until there is no quadratic
base satisfying the condition above. Then we make the quadratic section non-active, and consider the
remaining generalized equation where we remove doubles of all the quadratic coefficient bases. We
certainly have to remember that the bases that we removed express some variables in the quadratic
part (that became non-active) in the variables in the active part.} We have to add an equation to
$\Sigma_{v_{i+1}}$ that guarantees that the associated cyclic subgroups are generated by elements of
the same length. In this case the subgroup of $G_\Omega$ that is isomorphic to the coordinate group
of the new generalized equation in active part is a vertex group in an abelian splitting of
$G_\Omega$ and has smaller Delzant--Potyagailo complexity.

\paragraph{(b) QH-shortening}

\begin{lemma}
Suppose that the quadratic part of $\Omega$ does not correspond to the HNN-splitting of $G_\Omega$
(there are only quadratic coefficient bases), or, we cannot apply the entire transformation to the
quadratic part infinitely. In this case either $G_{\Omega}$ is a non-trivial free product or, applying
the automorphism of $G_\Omega$, one can replace the words corresponding to the quadratic bases in
the quadratic part by such words that in the new solution $H^+$ of $\Omega$ the length of the
quadratic part is bounded by some function $f_1(\Omega)$ times the length of the non-quadratic part.
Solution $H^+$ can be chosen consistent with $H$.
\end{lemma}
\begin{proof}
We apply entire transformation to the quadratic part. This corresponds to a sequence of Tietze
transformations of generators of $G_\Omega$. Since the entire transformation does not go infinitely,
some of the bases will eventually be transferred to the non-quadratic part (therefore, their length
is bounded by the length of the non-quadratic part). If not all of them are transferred, then
reincarnations of some of the other participating bases  must form matching pairs giving a
representation of $G_\Omega$ as a free product.
\end{proof}

If there is a matching pair, we replace $G_{\Omega}$ by the group obtained by removing a cyclic free
factor corresponding to a matching pair. We also replace $\Omega$ by the generalized equation
obtained by removing the matching pair. The Delzant--Potyagailo complexity decreases.

\paragraph{(c) Abelian splitting: short shift.}

\begin{proposition}
\label{PerSt}
Suppose $\Omega_v$ satisfies the following condition: the carrier base $\mu$ of the equation
$\Omega_v$ intersects with its dual $\overline\mu$ (form an overlapping pair) and is at least twice
longer than $|\alpha(\overline\mu) - \alpha(\mu)|$. Then $G_{\Omega_v}$ either splits as a
fundamental group of a graph of groups that has a free abelian vertex group or splits as an
HNN-extension with abelian associated subgroups.
\end{proposition}

The proof uses the technique of so-called periodic structures introduced by Razborov in his Ph.D
thesis and almost repeats the proof given in \cite{KMIrc} to show that the coordinate group of a
generalized equation splits in this case  as a fundamental group of a graph of groups that has a
free abelian vertex group or splits as an HNN-extension with abelian associated subgroups.
Nevertheless, we have to adjust the proof to the case of infinite non-Archimedean words. For
convenience of the reader we give a detailed proof of the proposition.

We construct an auxiliary equation ${\widehat\Omega_v}$ (which does not occur in $T(\Omega)$)
as follows. Firstly, we add a new constant non-active section $[\rho_v + 1, \rho_v + 2]$ to
the right of all sections in $\Omega_v$ (in particular, $h_{\rho_v + 1}$ is a new free variable).
Secondly, we introduce a new pair of bases $(\lambda, \overline\lambda)$ such that
$$\alpha(\lambda) = 1,\ \beta(\lambda) = \beta(\overline\mu),\ \alpha(\overline\lambda) = \rho_v
+ 1, \ \beta(\overline\lambda) = \rho_v + 2.$$
Notice that $\Omega_v$ can be obtained from ${\widehat\Omega_v}$ by (ET4): deleting
$\overline\lambda$ together with the closed section $[\rho_v + 1, \rho_v + 2]$. Let
$${\widehat\pi}_v : G_{\Omega_v} \rightarrow G_{\widehat\Omega_v}$$
be the isomorphism induced by (ET4). General JSJ-case still holds for $\widehat\Omega_v$, but now
$\lambda$ is the carrier base. Applying to $\widehat\Omega_v$ the entire transformation we obtain a
new vertex $\Omega_{v'}$ together with the isomorphism
$$\eta_{v'} : G_{\widehat\Omega_v} \rightarrow G_{\Omega_{v'}}.$$

As a culmination of the above process of adding auxiliary edges we get a generalized equation
$\Omega'$, which consists of finitely many closed sections $\sigma_1, \ldots, \sigma_p$ such that
there exists an overlapping pair $(\mu_i, \overline\mu_i)$ for each $\sigma_i,\ i \in [1,p]$.

\smallskip

Let $\sigma = \sigma_{i_0}$, where $i_0 \in [1,p]$ have the overlapping pair $(\mu,\overline\mu)$
such that
$$|[\alpha(\mu),\alpha(\overline\mu)]| = \max_{i \in [1,p]} |[\alpha(\mu_i),
\alpha(\overline\mu_i)]|.$$
Without loss of generality we can assume $i_0 = 1$, that is, $\sigma = [1, \beta(\overline\mu)]$.

\smallskip

A {\em periodic structure} on $\Omega'$ is a pair $\langle {\mathcal P}, R \rangle$, where
\begin{enumerate}
\item ${\mathcal P}$ is a set consisting of some items $h_i$, some bases $\lambda$, and some closed
sections $\sigma$ from $\Omega'$ such that the following conditions are satisfied:
\begin{itemize}
\item[(a)] if $h_i \in {\mathcal P}$ and $h_i \in \lambda$ then $\lambda \in {\mathcal P}$,
\item[(b)] if $\lambda \in {\mathcal P}$, then $\overline\lambda \in {\mathcal P}$,
\item[(c)] if $\lambda \in {\mathcal P}$ and $\lambda \in \sigma$, then $\sigma \in {\mathcal P}$,
\item[(d)] there exists a function ${\mathcal X}$ mapping the set of closed sections from
${\mathcal P}$ into $\{-1, +1\}$ such that for every $\lambda, \sigma_1, \sigma_2 \in {\mathcal P}$,
the condition that $\lambda \in \sigma_1$ and $\overline\lambda \in \sigma_2$ implies
$\varepsilon(\lambda) \cdot \varepsilon(\overline\lambda) = {\mathcal X}(\sigma_1) \cdot
{\mathcal X}(\sigma_2)$.
\end{itemize}
\item $R$ is an equivalence relation on the set of boundaries ${\mathcal B}$ belonging to closed
sections from ${\mathcal P}$ (if a boundary $l$ belongs to two closed section $\sigma_{left}(l) =
[i,l], \sigma_{right}(l) = [l,j]$ from ${\mathcal P}$ then instead of $l$ we add to ${\mathcal B}$
formal copies $l_{left}, l_{right}$ of $l$, which are boundaries of $\sigma_{left}(l)$ and
$\sigma_{right}(l)$ respectively), defined as follows: if $\lambda \in {\mathcal P}$ then
$$\alpha(\lambda) \sim_R \alpha(\overline\lambda),\ \beta(\lambda) \sim_R \beta(\overline\lambda)\
if\ \varepsilon(\lambda) = \varepsilon(\overline\lambda)$$
$$\alpha(\lambda) \sim_R \beta(\overline\lambda),\ \beta(\lambda) \sim_R \alpha(\overline\lambda)\
if\ \varepsilon(\lambda) = -\varepsilon(\overline\lambda).$$
\end{enumerate}

We call items, bases and sections {\em long} if they belong to ${\mathcal P}$ (the rest items, bases
and sections of $\Omega'$ are called {\em short} respectively).

\smallskip

A {\em period} is just a cyclically reduced $\Lambda$-word $u$. A reduced infinite word $w$ is called
{\em $u$-periodic} for a period $u$ if $w^{-1} \ast u \ast w = v$, where $|u| = |v|$. Since $u$ is
cyclically reduced it follows that either $w^{-1} \ast u = w^{-1} \circ u$, or $u \ast w = u \circ
w$. We are going to distinguish two cases of periodicity:
\begin{enumerate}
\item $w$ is {\em unbounded} $u$-periodic if $ht(w) \geqslant ht(u)$,
\item $w$ is {\em bounded} $u$-periodic if $ht(w) = ht(u)$ and $w = u^k \circ u_1$, where $k \geqslant
2$ and $u = u_1 \circ u_2$.
\end{enumerate}
Observe that in both cases $w$ begins with a power of $u$, but in the former case the power is
unbounded, while in the latter case it is bounded (hence the terms used).

\begin{lemma}
\label{product}
Let $w_1,w_2 \in CDR(\Lambda,Z)$, where $|w_1| \geqslant |w_2|$, be $u$-periodic for some period $u \in
CDR(\Lambda,Z)$. Assume that $w_1^{-1} \ast u \ast w_1 = v_1,\ w_2^{-1} \ast u \ast w_2 = v_2$. If
$w_1^{-1} \ast w_2$ is defined in $CDR(\Lambda,Z)$ then $w_2$ cancels completely in this product.
Moreover, $w_1^{-1} \ast w_2$ is $v_1^{-1}$-periodic.
\end{lemma}
\begin{proof} Follows from the definition of $u$-periodic words.
\end{proof}

\begin{lemma}
\label{double shift}
If $w \in CDR(\Lambda,Z)$ is both $u$- and $v$-periodic for some periods $u, v \in CDR(\Lambda,Z)$
then $[u,v] = 1$. In particular, $u$ is unbounded $v$-periodic if $ht(u) > ht(v)$.
\end{lemma}
\begin{proof} If $ht(u) = ht(v)$ then commutativity follows from Lemma 3.3 \cite{KhMS} since powers
of $u$ and $v$ have a common initial segment of length greater than $|u| + |v|$. If $ht(u) > ht(v)$
then from $v$-periodicity of $w$ it follows that $u$ is $v$-periodic. Now, since $w$ contains any
natural power of $u$ as an initial subword it follows that the terminal subword of $u$ of length
$|v|$ must coincide with $v$ and commutativity follows.
\end{proof}

\smallskip

Let $U$ be a solution of the equation $\Omega'$, where $h_i \to U_i$. According to our assumption,
$X_\mu$ is $P$-periodic, where $P = U[1, \alpha(\overline\mu)]$. Since $\mu$ and $\overline\mu$ are
products of items then every $U_i,\ i \in [\alpha(\overline\mu), \beta(\mu)]$, such that $|U_i|
\geqslant 2|P|$, is $Q_i$-periodic, where $Q_i$ is a $\Lambda$-word such that $|Q_i| = |P|$. Moreover,
we can associate with $U$ a periodic structure ${\mathcal P}(U ,P) = \langle {\mathcal P}, R
\rangle$ as follows. We construct ${\mathcal P}$ as a union of the chain of sets
$${\mathcal P_1} \subseteq {\mathcal P_1} \subseteq \cdots \subseteq {\mathcal P_k},$$
where $k \in \mathbb{N}$. To construct ${\mathcal P_0}$ we add to it $\sigma_1$ and all $Q$-periodic
items $h \in \sigma_1$, where $|Q| = |P|$, together with any base $\lambda$ (and its dual) containing
$h$. Now, assume that ${\mathcal P_i}$ is constructed. If there exists a base $\lambda \in {\mathcal
P_i}$ such that $\overline\lambda \in \sigma_j$ then we set
$${\mathcal P_{i+1}} = {\mathcal P_i} \cup \{\sigma_j\} \cup \{h, \lambda \in \sigma_j \mid
h\ {\rm is\ a\ Q-periodic\ item\ where}\ |Q| = |P|,\ \lambda\ {\rm is\ a\ base\ containing}\ h\}.$$
Since $\Omega'$ contains only finitely many closed sections then this process stops.

\begin{lemma}
\label{periodicity}
Let $U$ be a solution of the equation $\Omega'$ and ${\mathcal P}(U, P)$ be a periodic structure
corresponding to $U$. Then every base (item) $\lambda \in {\mathcal P}$ is $Q_\lambda$-periodic,
where $|Q_\lambda| = |P|$ and $Q_\lambda$ is conjugate to $P$ in $G$.
\end{lemma}
\begin{proof} It is enough to prove the statement for bases - the argument for items is similar. We
are going to use the induction on $k$, where $k$ is the number of sections in ${\mathcal P}$

\smallskip

By our assumption (short shift) we have $\mu \in {\mathcal P}$. It follows that any base $\lambda
\in \sigma_1$, such that $|\lambda| \geqslant 2|P|$, is $Q_\lambda$-periodic, where $|Q_\lambda| =
|P|$. Moreover, if $w = U[\alpha(\overline\mu),\alpha\lambda]$ then $w \in G$ and $w^{-1} \ast P
\ast w = Q_\lambda$. Hence, all bases from ${\mathcal P_1}$ have the properties stated in the lemma.

\smallskip

Suppose we have the statement of the lemma for all bases from ${\mathcal P_{k-1}}$. By definition,
${\mathcal P_k}$ is obtained from ${\mathcal P_{k-1}}$ by adding a section $\sigma_k$, such that
there exists $\lambda \in {\mathcal P_{k-1}}$ and $\overline\lambda \in \sigma_k$, as well as all
appropriate items and bases of $\sigma_k$. Observe that since $\mu_k$ is $Q$-periodic, where $Q =
U[\alpha(\mu_k), \alpha(\overline\mu_k)]$, it follows that $\overline\lambda$ is $Q'$-periodic,
where $|Q'| = |Q|$ and $Q'$ is conjugate to $Q$ by means of $g \in G$. At the same time, by the
induction hypothesis, $\lambda$ is $Q_\lambda$-periodic, where $|Q_\lambda| = |P|$. Hence, by Lemma
\ref{double shift} we have $[Q', Q_\lambda] = 1$ and it follows that $[Q, g^{-1} \ast
Q_\lambda \ast g] = 1$, where $|g^{-1} \ast Q_\lambda \ast g| = |P|$. Now, since $|P| \geqslant
|Q|$ and $|\mu_k| \geqslant |\lambda|$ it follows that $\mu_k$ is $g^{-1} \ast Q_\lambda \ast g$-periodic.
The statement of the lemma for all bases in $\sigma_k$ of the length at least $2 |P|$ follows
automatically now.

\end{proof}

\begin{remark}
\label{rem:period}
From the construction of ${\mathcal P}(U, P)$ and Lemma \ref{periodicity} it follows that if
$\sigma_i \in {\mathcal P}$ with an overlapping pair $(\mu_i,\overline\mu_i)$ and if $l \in [\alpha
(\sigma_i),\beta(\sigma_i)]$ then there exists a period $Q_l$ conjugate in $G$ either to $P$ or
$P^{-1}$ such that
\begin{enumerate}
\item[(i)] if $ht([\alpha(\sigma_i),l]) > ht(P)$ then $U[\alpha(\sigma_i),l]$ has any natural power
of $Q_l$ as a terminal subword,
\item[(ii)] if $ht([l,\beta(\sigma_i)]) > ht(P)$ then $U[1,\beta(\sigma_i)]$ has any natural power
of $Q_l$ as an initial subword,
\item[(iii)] if $ht([\alpha(\sigma_i),l]) \leqslant ht(P)$ then $U[\alpha(\sigma_i),l] = Q'' \circ
Q_l^{k_l}$, where $Q_l = Q' \circ Q''$,
\item[(iv)] if $ht([l,\beta(\sigma_i)]) \leqslant ht(P)$ then $U[1,\beta(\sigma_i)] = Q_l^{k_l} \circ
Q'$, where $Q_l = Q' \circ Q''$.
\end{enumerate}
\end{remark}

\begin{lemma}
\label{le:PP}
${\mathcal P}(U, P)$ corresponding to a solution $U$ of $\Omega'$ is a periodic structure on
$\Omega'$.
\end{lemma}
\begin{proof} Parts 1(a), 1(b) and 1(c) of the definition of periodic structure hold by construction.

\smallskip

Set ${\mathcal X}(\sigma_1) = 1$. If there is a base $\lambda \in \sigma_1$ such that $\overline
\lambda \in \sigma_2$ then we set ${\mathcal X}(\sigma_2) = \varepsilon(\lambda) \varepsilon
(\overline\lambda)$. Observe that the sign of ${\mathcal X}(\sigma_2)$ is well-defined since if
there exists another base $\lambda_1 \in \sigma_1$ such that $\overline\lambda_1 \in \sigma_2$ and
$\varepsilon(\lambda) \varepsilon(\overline\lambda) \neq \varepsilon(\lambda_1) \varepsilon(\overline
\lambda_1)$ then it follows that there exists a group element conjugating $P$ into $P^{-1}$, which is
impossible. Hence 1(d) holds.

\smallskip

By Remark \ref{rem:period} if two boundaries $l_1, l_2$ belong to sections from ${\mathcal P}$ then
there exist $Q_{l_1}, Q_{l_2}$ conjugate in $G$ either to $P$ or $P^{-1}$ such that some of the
conditions listed in the remark hold. Thus, we define a relation $R$ as follows: $l_1 \sim_R l_2$ if
and only if $Q_{l_1} = Q_{l_2}$, and part 2 follows immediately.

\end{proof}

Now let us fix a nonempty periodic structure $\langle {\mathcal P}, R \rangle$. Item (d) allows us
to assume (after replacing the variables $h_i, \ldots, h_{j-1}$ by $h_{j-1}^{-1}, \ldots, h_i^{-1}$
on those sections $[i,j] \in {\mathcal P}$ for which ${\mathcal X}([i,j]) = -1$) that $\varepsilon
(\mu) = 1$ for all $\mu \in {\mathcal P}$. For a boundary $k$, we will denote by $(k)$ the
equivalence class of the relation $R$ to which it belongs.

\smallskip

Let us construct an oriented graph $\Gamma$ whose set of vertices is the set of $R$-equivalence
classes. For each long $h_k$ we introduce an oriented edge $e$ leading from $(k)$ to $(k+1)$ and an
inverse edge $e^{-1}$ leading from $(k+1)$ to $(k)$. This edge $e$ is assigned the label $h(e) =
h_k$ (respectively, $h(e^{-1}) = h_k^{-1}$). For every path $r = e_1^{\pm 1} \cdots e_s^{\pm 1}$ in
the graph $\Gamma$, we denote by $h(r)$ its label $h(e_1^{\pm 1}) \cdots h(e_j^{\pm 1})$.

The periodic structure $\langle {\mathcal P}, R \rangle$ is called {\em connected}, if the graph
$\Gamma$ is connected. Suppose first that $\langle {\mathcal P}, R \rangle$ is connected. We can
also suppose that each boundary of $\Omega'$ is a boundary between two bases.

\begin{lemma}
\label{paths}
Let $U$ be a solution of $\Omega'$ and $\langle {\mathcal P}, R \rangle = {\mathcal P}(U, P)$. If
$p$ is a path in the graph $\Gamma$ from $(i)$ to $(j)$ then $U[p]^{-1} \ast Q_i \ast U[p] = Q_j$
in $G$.
\end{lemma}
\begin{proof} If $p = h_{i_1} \cdots h_{i_k}$ then all products between $U_j,\ j \in \{i_1, \ldots,
i_k\}$ are defined and the required result follows from Lemma \ref{product}.
\end{proof}

\begin{lemma}
\label{cycles}
Let $U$ be a solution of $\Omega'$ and $\langle {\mathcal P}, R \rangle = {\mathcal P}(U, P)$. If
$c_1, c_2$ are cycles in the graph $\Gamma$ at the vertex $v = (k)$ then $[U[c_1], U[c_2]] = 1
$ in $G$.
\end{lemma}
\begin{proof} Assume at first that $(k) = (1)$. There exists a loop $c$ at $(1)$ such that $U[c] =
P$. By Lemma \ref{paths} $U[p]^{-1} \ast P \ast U[p] = P$ in $G$ for every path $p$ from $(1)$ to
$(1)$, so $[U[c_1], P] = [U[c_2], P] = 1$ and from the CSA-property of $G$ the required follows.

If $(k) \neq (1)$ then, since $\Gamma$ is connected it follows that there exists a path $p$ from $v$
to $v_0$. Moreover, by Lemma \ref{paths} we have $U[p]^{-1} \ast Q_k \ast U[p] = P$ in $G$. So, by
the CSA-property commutativity of cycles at $(k)$ follows from commutativity of cycles at $(1)$
which is already proved.

\end{proof}

$\Omega'$ is called {\em periodized} with respect to a given periodic structure $\langle {\mathcal
P}, R \rangle$ if for every two cycles $c_1$ and $c_2$ with the same initial vertex in the graph
$\Gamma$ there is a relation $[h(c_1), h(c_2)] = 1$ in $G_{\Omega'}$. Observe that according to Lemma
\ref{cycles} $\Omega'$ is periodized with respect to ${\mathcal P}(U, P)$ for any solution $U$.

\smallskip

Let $\Gamma_0$ be the subgraph of the graph $\Gamma$ having the same set of vertices and consisting
of the edges $e$ whose labels do not belong to ${\mathcal P}$. Choose a maximal sub-forest $T_0$ in
the graph $\Gamma_0$ and extend it to a maximal sub-forest $T$ of the graph $\Gamma$. Since $\langle
{\mathcal P}, R \rangle$ is connected by assumption, it follows that $T$ is a tree. Let $v_0$ be an
arbitrary vertex of the graph $\Gamma$ and $r(v_0, v)$ the (unique) path from $v_0$ to $v$ all of
whose vertices belong to $T$. For every edge $e = (v, v')$ not lying in $T$, we introduce a cycle
$c_e = r(v_0, v) e (r(v_0, v'))^{-1}$. Then the fundamental group $\pi_1(\Gamma, v_0)$ is generated
by the cycles $c_e$ (see, for example, the proof of Proposition 3.2.1 \cite{LS}).

Furthermore, the set of elements
\begin{equation}
\label{2.52}
\{h(e) \mid e \in T \} \cup \{h(c_e) \mid e \not \in T \}
\end{equation}
forms a basis of the free group with the set of generators $\{h_k \mid h_k \in \sigma\}$. If $\lambda
\in {\mathcal P}$, then $(\beta(\lambda)) = (\beta(\overline\lambda)),\ (\alpha(\lambda)) = (\alpha
(\overline\lambda))$ by the definition of $R$ and, consequently, the word
$$h[\alpha(\lambda), \beta(\lambda)]\ h[\alpha(\overline\lambda), \beta(\overline\lambda)]^{-1}$$
is the label of a cycle $c'(\lambda)$ from $\pi_1(\Gamma,(\alpha(\lambda)))$. Let $c(\lambda) =
r(v_0, (\alpha(\lambda)))\ c'(\lambda)\ r(v_0,(\alpha(\lambda)))^{-1}$. Then
\begin{equation}
\label{2.53}
h(c(\lambda)) = u\ h[\alpha(\lambda), \beta(\lambda)]\ h[\alpha(\overline\lambda), \beta(\overline
\lambda)]^{-1}\ u^{-1},
\end{equation}
where $u$ is a certain word. Since $c(\lambda) \in \pi_1(\Gamma, v_0)$, it follows that $c(\lambda)
= b_\lambda (\{c_e \mid e \not \in T \})$, where $b_\lambda$ is a certain word in the indicated
generators which can be effectively constructed (see Proposition 3.2.1 \cite{LS}).

Let $\tilde{b}_\lambda$ denote the image of the word $b_\lambda$ in the abelianization of $\pi(\Gamma,
v_0)$. Denote by $\widetilde{Z}$ the free abelian group consisting of formal linear combinations
$\sum_{e \not \in T} n_e \tilde{c}_e\ \ (n_e \in {\mathbb Z})$, and by $\widetilde{B}$ its subgroup
generated by the elements $\tilde{b}_\lambda\ \ (\lambda \in {\mathcal P})$ and the elements
$\tilde{c}_e\ \ (e \not \in T, \ h(e) \not \in {\mathcal P})$. Let $\widetilde{A} = \widetilde{Z} /
\widetilde{B},\ T(\widetilde{A})$ the torsion subgroups of the group $\widetilde{A}$, and
$\widetilde{Z}_1$ the preimage of $T(\widetilde{A})$ in $\widetilde{Z}$. The group $\widetilde{Z} /
\widetilde{Z}_1$ is free, therefore, there exists a decomposition of the form
\begin{equation}
\label{2.54}
\widetilde{Z} = \widetilde{Z}_1 \oplus \widetilde{Z}_2,\ \widetilde{B} \subseteq \widetilde{Z}_1,\
|\widetilde{Z}_1 : \widetilde{B}| < \infty.
\end{equation}
Note that it is possible to express effectively a certain basis $\widetilde{\overline{c}}^{(1)},\
\widetilde{\overline{c}}^{(2)}$ of the group $\widetilde{Z}$ in terms of the generators
$\widetilde{c}_e$ so that for the subgroups $\widetilde{Z}_1$ and $\widetilde{Z}_2$ generated
respectively by $\widetilde{\overline{c}}^{(1)}$ and $\widetilde{\overline{c}}^{(2)}$, the relation
(\ref{2.54}) holds. It suffices, for instance, to look through the bases one by one, using the fact
that under the condition $\widetilde{Z} = \widetilde{Z}_1 \oplus \widetilde{Z}_2$ the relations
$\widetilde{B} \subseteq \widetilde{Z}_1,\ |\widetilde{Z}_1 : \widetilde{B}| < \infty$ hold if and
only if the generators of the groups $\widetilde{B}$ and $\widetilde{Z}_1$ generate the same linear
subspace over ${\bf Q}$, and the latter is easily verified algorithmically. Notice, that a more
economical algorithm can be constructed by analyzing the proof of the classification theorem for
finitely generated abelian groups. By Proposition 1.4.4  \cite{LS}, one can effectively construct a
basis $\overline{c}^{(1)}$, $\overline{c}^{(2)}$ of the free (non-abelian) group $\pi_1(\Gamma,
v_0)$ so that $\widetilde{\overline{c}}^{(1)}$, $\widetilde{\overline{c}}^{(2)}$ are the natural
images of the elements $\overline{c}^{(1)}$, $\overline{c}^{(2)}$ in $\widetilde{Z}$.

Now assume that $\langle {\mathcal P}, R \rangle$ is an arbitrary periodic structure on $\Omega'$,
not necessarily connected. Let $\Gamma_1, \ldots, \Gamma_r$ be the connected components of the graph
$\Gamma$. The labels of edges of the component $\Gamma_i$ form in $\Omega'$ a union of closed sections
from ${\mathcal P}$. Moreover, if a base $\lambda \in {\mathcal P}$ belongs to such a section, then
its dual $\overline\lambda$ also does by definition of periodic structure. Therefore, by taking for
${\mathcal P}_i$ the set of labels of edges from $\Gamma_i$ belonging to ${\mathcal P}$, sections
to which these labels belong, and bases $\lambda \in {\mathcal P}$ belonging to these sections, and
restricting in the corresponding way the relation $R$, we obtain a periodic connected structure
$\langle {\mathcal P}_i, R_i \rangle$ with the graph $\Gamma_i$.

The notation $\langle {\mathcal P}', R' \rangle \subseteq \langle {\mathcal P}, R \rangle$ means
that ${\mathcal P}' \subseteq {\mathcal P}$ and the relation $R'$ is a restriction of the relation
$R$. In particular, $\langle {\mathcal P}_i, R_i \rangle \subseteq \langle {\mathcal P}, R \rangle$
in the situation described in the previous paragraph. Since $\Omega'$ is periodized, the periodic
structure must be connected.

Let $e_1, \ldots, e_m$ be all the edges of the graph $\Gamma$ from $T - T_0$. Since $T_0$ is the
spanning forest of the graph $\Gamma_0$, it follows that $h(e_1), \ldots, h(e_m) \in {\mathcal P}$.
Let $F(\Omega')$ be a free group generated by the variables of $\Omega'$. Consider in the group
$F(\Omega')$ a new basis
$$Y = \{\overline t,\ \{h(e) \mid e \in T\},\ h(\overline c^{(1)}),\ h(\overline c^{(2)})\},$$
where variables $\bar t$ do not belong to the closed sections from $\mathcal P$. Let $v_i$ be the
initial vertex of the edge $e_i$. We introduce new variables $\overline u^{(i)} = \{u_{ie} \mid e
\notin T,\ e \notin {\mathcal P}\},\ \overline z^{(i)} = \{z_{ie} \mid e \notin T,\ e \notin
{\mathcal P}\}$ for $1 \leqslant i \leqslant m$, as follows
\begin{equation}
\label{2.59}
u_{ie} = h(r(v_0,v_i))^{-1}\ h(c_e)\ h(r(v_0,v_i)),
\end{equation}
\begin{equation}
\label{2.60}
h(e_i)^{-1}\ u_{ie}\ h(e_i) = z_{ie}.
\end{equation}

Without loss of generality we can assume $v_0 = (1)$.

\begin{lemma}
\label{2.10''}
Let $\Omega'$ be periodized with respect to a periodic structure $\langle {\mathcal P}, R \rangle$.
Let $K$ be the subgroup of $G_{\Omega'}$ generated by
$$Y_0 = \{ \overline t,\ \{h(e) \mid e \in T_0\},\ h(\overline c^{(1)}),\ \overline u^{(i)},
\overline z^{(i)}, i \in [1,m]\}.$$
If $|\overline c^{(2)}| = s \geqslant 1$ then the group $G_{\Omega'}$ splits as a fundamental group
of a graph of groups with two vertices, where one vertex group is $K$ and the other is a free abelian
group generated by $h(\overline c^{(2)}),\ h(\overline c^{(1)})$. The corresponding edge group is
generated by $h(\overline c^{(1)})$. The other edges are loops at the vertex with vertex group $K$,
have stable letters $h(e_i),\ i \in [1,m]$, and associated subgroups $\langle \overline u^{i}
\rangle,\ \langle \overline z^{i} \rangle$. If $\overline c^{(2)} = \emptyset$ then there is no
vertex with abelian vertex group.
\end{lemma}
\begin{proof} We are going to study in more detail how the variables $h(e_i),\ i \in [1,m]$ can
participate in the equations from $\Omega'$ rewritten in the set of variables $Y$.

If $h_k$ does not lie on a closed section from ${\mathcal P}$, or $h_k \notin {\mathcal P}$, but $e
\in T$ (where $h(e) = h_k$), then $h_k$ belongs to the basis $Y$ and is distinct from each of
$h(e_1), \ldots, h(e_m)$. Now let $h(e) = h_k$, $h_k \notin {\mathcal P}$ and $e \notin T$. Then
$e = r_1 c_e r_2$, where $r_1, r_2$ are paths in $T$. Since $e \in \Gamma_0,\ h(c_e)$ belongs to
$\langle c^{(1)} \rangle$ modulo commutation of cycles. The vertices $(k)$ and $(k+1)$ lie in the
same connected component of the graph $\Gamma_0$ and, hence, they are connected by a path $s$ in
the forest $T_0$. Furthermore, $r_1$ and $s\ r_2^{-1}$ are paths in the tree $T$ connecting the
vertices $(k)$ and $v_0$. Consequently, $r_1 = s\ r_2^{-1}$. Thus, $e = s\ r_2^{-1}\ c_e\ r_2$ and
$h_k = h(s)\ h(r_2)^{-1}\ h(c_e)\ h(r_2)$. The variable $h(e_i),\ i \in [1,m]$ can occur in the
right-hand side of the expression obtained (written in the basis $Y$) only in $h(r_2)$ and at most
once. Moreover, the sign of this occurrence (if it exists) depends only on the orientation of the
edge $e_i$ with respect to the root $v_0$ of the tree $T$. If $r_2 = r_2'\ e_i^{\pm 1}\ r_2''$ then
all the occurrences of the variable $h(e_i)$ in the words $h_k$ written in the basis $Y$, with $h_k
\notin {\mathcal P}$, are contained in the occurrences of words of the form $h(e_i)^{\mp 1}\
h((r_2')^{-1}\ c_e\ r_2')\ h(e_i)^{\pm 1}$, that is, in occurrences of the form $h(e_i)^{\mp 1}\
h(c)\ h(e_i)^{\pm 1}$, where $c$ is a certain cycle of the graph $\Gamma$ starting at the initial
vertex of the edge $e_i^{\pm 1}$.

Therefore all the occurrences of $h(e_i),\ i \in [1,m]$ in the equations corresponding to $\lambda
\notin {\mathcal P}$ are of the form $h(e_i^{-1})\ h(c)\ h(e_i)$. Also, $h(e_i)$ does not occur in
the equations corresponding to $\lambda \in {\mathcal P}$ in the basis $Y$. The system $\Omega'$ is
equivalent to the following system in the variables $Y$:
$$u_{ie} = h(r(v_0,v_i))^{-1}\ h(c_e)\ h(r(v_0,v_i)),$$
$$h(e_i)^{-1}\ u_{ie}\ h(e_i) = z_{ie},$$
$$[u_{ie_1}, u_{ie_2}] = 1,$$
$$[h(c_1),h(c_2)] = 1, \ c_1,c_2 \in c^{(1)}, c^{(2)},$$
$$\overline\psi(Y_0) = 1,$$
where $\overline\psi(Y_0)$ does not contain $h(e_i),\ \bar c^{(2)}$. Let $K = G_{\overline\psi}$.
Then to obtain $G_\Omega$ we first take an HNN-extension of the group $K$ with abelian associated
subgroups generated by $\overline u^{(i)}$ and $\overline z^{(i)}$ and stable letters $h(e_i)$, and
then extend the centralizer of the image of $\langle \overline c^{(1)} \rangle$ by the free abelian
subgroup generated by the images of $\overline c^{(2)}$.

\end{proof}
Proposition \ref{PerSt} now follows from Lemma \ref{2.10''}.

\paragraph{(d) Abelian splitting: long shift.} If $\Omega$ does not satisfy the conditions of
(a)---(c), we perform QH-shortening, then apply the entire transformation and then, if possible, the
transformation (D7).

\begin{lemma}
\label{3.2}
Let
$$v_1 \rightarrow v_2 \rightarrow \cdots \rightarrow v_r \rightarrow \cdots$$
be an infinite path in $T(\Omega)$. Then there exists a natural number $N$ such that all the
generalized equations in vertices $v_n,\ n \geqslant N$ satisfy the general JSJ-case {(d)}.
\end{lemma}
\begin{proof} Indeed, the Tietze cleaning either replaces the group by its proper free factor or
decreases the complexity. Every time when the case (a) holds we replace $G$ by some vertex group in
a non-trivial abelian splitting of $G$. This can be done only finitely many times \cite{DP}. Every
time when case (c) takes place, we decrease the complexity.
\end{proof}

\begin{proposition}\label{(d)}
The general JSJ case (d) cannot be repeated infinitely many times.
\end{proposition}
\begin{proof}
Consider an infinite path
\begin{equation}
\label{3.6}
r = v_1 \rightarrow v_2 \rightarrow \cdots \rightarrow v_m \cdots
\end{equation}
We have $\tau_{v_i}' = \tau_v'$.

Denote by $\mu_i$ the carrier base of the equation $\Omega_{v_i}$. The path (\ref{3.6}) will be
called $\mu$-reducing if $\mu_1 = \mu$, and either $\mu_2$ does not overlap with its double or they
overlap but $|\mu|\leqslant 2|\alpha (\mu )-\alpha (\bar\mu )|$, and
 $\mu$ occurs in the sequence $\mu_1, \ldots, \mu_{m-1}$ at least twice.

The path (\ref{3.6}) is called {\em prohibited} if it can be represented in the form
\begin{equation}
\label{3.7}
r = r_1 s_1 \cdots r_l s_l r',
\end{equation}
such that for some sequence of bases $\eta_1, \ldots, \eta_l$ the following three properties hold:
\begin{enumerate}
\item[(1)] every base occurring at least once in the sequence $\mu_1, \ldots, \mu_{m-1}$ occurs at
least $4n(1+f_1(\Omega_{v_2}))$  times in the sequence $\eta_1,
\ldots, \eta_l$, where $n$ is the number of pairs of bases in equations $\Omega_{v_i}$,
\item[(2)] the path $r_i$ is $\eta_i$-reducing,
\item[(3)] every transfer base of some equation of path $r$ is a transfer base of some equation of
path $r'$.
\end{enumerate}
Every infinite path contains a prohibited subpath. Indeed, let $\omega$ be the set of all bases
occurring in the sequence $\mu_1, \ldots, \mu_m, \ldots$ infinitely many times, and $\widetilde\omega$
the set of all bases, that are transfer bases of infinitely many equations $\Omega_{v_i}$. If one
cuts out some finite part in the beginning of this infinite path, one can suppose that all the bases
in the sequence $\mu_1, \ldots, \mu_m, \ldots$ belong to $\omega$ and each base that is a transfer
base of at least one equation, belongs to $\widetilde\omega$. Such an infinite path for any $\mu \in
\omega$ contains infinitely many non-intersecting $\mu$-reducing finite subpaths. Hence, it is
possible to construct a subpath (\ref{3.7}) of this path satisfying the first two conditions in the
definition of a prohibited subpath. Making $r'$ longer one obtains a prohibited subpath.

Suppose we have a prohibited path (\ref{3.6}). As before, let $\mu_i$ denote the carrier base of
$\Omega_{v_i}$, and $\omega = \{\mu_1, \ldots, \mu_{m-1}\}$ and $\widetilde\omega$ denote the set of
such bases which are transfer bases for at least one equation in (\ref{3.6}). By $\omega_1$ denote
the set of such bases $\mu$ for which either $\mu$ or $\overline\mu$ belongs to $\omega \cup
\widetilde\omega$. By $\omega_2$ denote the set of all the other bases. Let
$$\alpha(\omega) = \min(\min_{\mu \in \omega_2} \alpha(\mu),j),$$
where $j$ is the boundary between active and non-active sections.

Let
$$X_{\mu} \circeq U[\alpha(\mu), \beta(\mu)].$$
If $(\Omega_{v_i}, U)$ corresponds to a vertex in the sequence (\ref{3.6}) then denote
\begin{equation}
\label{3.12}
|U_\omega| = \sum_{i = 1}^{\alpha(\omega)-1} |U_i|,
\end{equation}

\begin{equation}
\label{3.13}
\psi_{\omega}(U) = \sum_{\mu \in \omega_1} |X_{\mu}| - 2 |U|_{\omega}.
\end{equation}

We call $\psi_{\omega}(U)$ the {\em excess} of $\Omega$. Since the case (a) is not applicable,
$\psi_{\omega}(U) \geqslant 0$ and is comparable with the length of the section $[1, \alpha(\omega)]$.
Consider the quadratic part of $\widetilde\Omega_{v_1}$ which is situated to the left of
$\alpha(\omega)$. Let $A$ be the length of the quadratic part and $B$ be the length of the
non-quadratic part.

\begin{lemma}
Suppose there exists a function $f_1(\Omega_{v_1})$ such that we have an inequality
\begin{equation}
\label{quadr}
A \leqslant B f_1.
\end{equation}
Then there is a function $f_2(\Omega_{v_1})$ such that
\begin{equation}
\label{psi}
|U_{\omega}| \leqslant \psi_{\omega}(U) f_2.
\end{equation}
\end{lemma}
\begin{proof} $|U_{\omega}| = A + B \leqslant (1+f_1) B \leqslant (1+f_1) \psi_{\omega}(U)$, because
 $B \leqslant \psi_{\omega}(U)$.
\end{proof}

From the definition of the process it follows that all the words $U^{(i)}[1, \rho_i + 1]$ are the
ends of the word $U^{(1)}[1, \rho_1 + 1]$, that is
\begin{equation}
\label{3.20}
U^{(1)}[1, \rho_1 + 1] \doteq u_i U^{(i)}[1, \rho_i + 1].
\end{equation}
On the other hand any base $\mu \in \omega_2$ participates in these transformations neither as a
carrier base, nor as a transfer base. Hence, $U^{(1)}[\alpha(\omega), \rho_1 + 1]$ is the end of the
word $U^{(i)}[1, \rho_i + 1]$, that is,
\begin{equation}
\label{3.21}
U^{(i)}[1, \rho_i + 1] \doteq v_i U^{(1)}[\alpha(\omega), \rho_1 + 1].
\end{equation}
So we have
\begin{equation}
\label{3.22}
|U^{(i)}_\omega| - |U^{(i+1)}_\omega| = |v_i| - |v_{i+1}| = |u_{i+1}| - |u_i| = |X_{\mu_i}^{(i)}|
- |X_{\mu_i}^{(i+1)}|.
\end{equation}
In particular (\ref{3.13}), (\ref{3.22}) imply that $\psi_\omega(U^{(1)}) = \psi_\omega(U^{(2)}) =
\ldots = \psi_\omega(U^{(m)}) = \psi_\omega$. Denote the number (\ref{3.22}) by $\delta_i$.

Let the path (\ref{3.6}) be $\mu$-reducing. Estimate $d(U_m) = \sum_{i=1}^{m-1} \delta_i$ from below.
First notice that if $\mu_{i_1} = \mu_{i_2} = \mu$ for $i_1 < i_2$ and $\mu_i \neq \mu$ for $i_1 < i
< i_2$ then
\begin{equation}
\label{3.23}
\sum_{i = i_1}^{i_2 - 1} \delta_i \geqslant |U^{i_1 + 1}[1, \alpha(\overline\mu_{i_1 + 1})]|.
\end{equation}
Indeed, if $i_2 = i_1 + 1$ then $\delta_{i_1} = d(U^{(i_1)}[1, \alpha(\overline\mu)] =
d(U^{(i_1 + 1)}[1, \alpha(\overline\mu)]$. If $i_2 > i_1 + 1$ then $\mu_{i_1 + 1} \neq \mu$ and $\mu$
is a transfer base in the equation $\Omega_{v_{i_1 + 1}}$. Hence, $\delta_{i_1 + 1} + |U^{(i_1 + 2)}
[1, \alpha(\mu)]| = |U^{(i_1 + 1)}[1, \alpha(\mu_{i_1 + 1})]|$. Now (\ref{3.23}) follows from
$$\sum_{i = i_1 + 2}^{i_2 - 1} \delta_i \geqslant |U^{(i_1 + 2)}[1, \alpha(\mu)]|.$$
So, if the bases $\mu_2$ and $\overline \mu_2$ do not intersect in the equation $\Omega_{v_2}$ then
(\ref{3.23}) implies that
$$\sum_{i = 1}^{m - 1} \delta_i \geqslant |U^{(2)}[1, \alpha(\overline\mu_2)]| \geqslant
|X_{\mu_2}^{(2)}| \geqslant |X_\mu^{(2)}| = |X_\mu^{(1)}| - \delta_1,$$
which implies that
\begin{equation}
\label{3.24}
\sum_{i = 1}^{m - 1} \delta_i \geqslant \frac{1}{2}|X_\mu^{(1)}|.
\end{equation}
We obtain the same inequality if $\mu_2$ overlaps with $\overline\mu_2$ but $|\mu_2| \leqslant
2|\alpha(\mu_2) - \alpha(\overline\mu_2)|.$

Suppose now that the path (\ref{3.6}) is prohibited. Hence, it can be represented in the form
(\ref{3.7}). From  the definition (\ref{3.13}) we have $\sum_{\mu \in \omega_1} d(X_\mu^{(m)})
\geqslant \psi_\omega$. So, at least for one base $\mu \in \omega_1$ the inequality $d(X_\mu^{(m)})
\geqslant \frac{1}{2n} \psi_\omega$ holds. Because $X_\mu^{(m)} \doteq (X_{\overline\mu}^{(m)})^{\pm 1}$,
we can suppose that $\mu \in \omega \cup \widetilde{\omega}$. Let $m_1$ be the length of the path
$r_1 s_1 \cdots r_l s_l$ in (\ref{3.7}). If $\mu \in \widetilde{\omega}$ then by the third part of
the definition of a prohibited path there exists $m_1 \leqslant i \leqslant m$ such that $\mu$ is a
transfer base of $\Omega_{v_i}$. Hence, $d(X_{\mu_i}^{(m_1)}) \geqslant d(X_{\mu_i}^{(i)}) \geqslant
d(X_\mu^{(i)}) \geqslant d(X_\mu^{(m)}) \geqslant \frac{1}{2n} \psi_\omega$. If $\mu \in \omega$ then take
$\mu$ instead of $\mu_i$. We proved the existence of a base $\mu \in \omega$ such that
\begin{equation}
\label{3.27}
|X_\mu^{(m_1)}| \geqslant \frac{1}{2n} \psi_\omega.
\end{equation}
By the definition of a prohibited path, the inequality (\ref{3.24}) and the inequality $d(X_\mu^{(i)})
\geqslant d(X_\mu^{(m_1)})\ \ (1 \leqslant i \leqslant m_1)$, we obtain
\begin{equation}
\label{3.28}
\sum_{i = 1}^{m_1 - 1} \delta_i \geqslant \max\left\{\frac{1}{4n} \psi_\omega, 1\right\}(4n f_2).
\end{equation}
By (\ref{3.22}) the sum in the left part of the inequality (\ref{3.28}) equals $d_\omega(U^{(1)}) -
d_\omega(U^{(m_1)})$. Hence,
$$|U^{(1)}_\omega| \geqslant \max\left\{\frac{1}{4n} \psi_\omega, 1\right\} (4n f_2) = \psi_\omega f_2,$$
which, in case of the inequality (\ref{quadr}), contradicts (\ref{psi}). This contradiction was
obtained from the supposition that there are prohibited paths (\ref{3.6}). Hence there are no
prohibited paths. The proposition is proved.

\end{proof}

{\bf Proof of Theorem \ref{th:main1}.}  We perform the elimination process for the generalized
equation $\Omega = \Omega_{v_0}$ corresponding to $G$. If the process goes infinitely we obtain one
of the following:
\begin{itemize}
\item free splitting of $G$ with at least on non-trivial free factor (4.4.1) and, maybe, some surface
group factor (4.4.2),

\item a decomposition of $G$ as the fundamental group of a graph of groups with QH-vertex groups
corresponding to quadratic sections (4.4.4 (a)),

\item decomposition of $G$ as the fundamental group of a graph of groups with abelian vertex groups
corresponding to periodic structures ( see Proposition \ref{PerSt} in 4.4.4 (c)) and HNN-extensions
with stable letters infinitely longer than the generators of the abelian associated subgroups
(4.4.4 (c)).
\end{itemize}
Then we continue the elimination process with the generalized equation $\Omega_{v_0}$ where the
active part corresponds to weakly rigid subgroups. In the case of an HNN-extension the complexity
$\tau(\Omega_{v_1})$ is smaller than $\tau(\Omega_{v_0}).$ In the other cases the Delzant--Potyagailo
complexity of weakly rigid subgroups is smaller by Proposition \ref{DP}. Therefore this procedure
stops. At the end we obtain some generalized equation $\Omega_{v_{fin}}$ with all non-active sections.
Continuing the elimination process till the end, we obtain a complete system of linear equations with
integer coefficients $\Sigma_{complete}$ on the lengths of  items $h_i$'s that is automatically
satisfied, therefore the associated maximal abelian subgroups  are length-isomorphic.

Notice that if $G_{i+1} = G_{\Omega _{v_{i+1}}}$ is obtained  as the fundamental group of a graph of
groups with two vertices, where $G_i$ is one vertex group and the other vertex is a QH-vertex, then
one can also obtain $G_{i+1}$ from $G_i$  by taking a free product of $G_i$ and a free group and then
the HNN-extension with cyclic associated subgroups. If $G_{i+1}$ is an amalgamated product of $G_i$
and a free abelian group then it can be obtained from $G_i$ as a series of extensions of centralizers.

This completes the proof of Theorem \ref{th:main1}.

\hfill $\square$

\begin{remark}
\label{nonreg1}
If we begin with the group $\widetilde G$ mentioned in Remark \ref{nonreg} with free but not necessary
regular length function in $\Lambda$ then in the Elimination process we work with the generalized
equation $\Omega$ and add a finite number of elements from $\widehat G$. Thus we got an embedding of
$\widetilde G$ in a group that can be represented as a union of a finite series of groups
$$G_1 < G_2 < \cdots < G_n = G,$$
where
\begin{enumerate}
\item $G_1$ is a free group,
\item $G_{i+1}$ is obtained from $G_i$ by finitely many HNN-extensions in which associated subgroups
are maximal abelian and length isomorphic as subgroups of $\Lambda$.
\end{enumerate}
\end{remark}

\begin{remark}
\label{fin}
As a result of the Elimination process, the equation $\Omega_{v_{fin}}$ (we will denote it
$\Omega_{fin}$) is defined on the multi-interval $I$, that is, a union of closed sections which have
a natural hierarchy: a section $\sigma_1$ is smaller than a section $\sigma_2$ if the largest base
on $\sigma_2$ is infinitely larger than the largest base on $\sigma_1$.

The lengths of bases satisfy the system of linear equations $\Sigma_{complete}$.
\end{remark}

\section{Proof of Theorems \ref{th:main3} and \ref{th:main4}}

Suppose $G$, as above, is a $\Lambda$-free group with a regular action. Let $\Omega$ be a generalized
equation for $G$ corresponding to the union of closed sections $I$, $G = \langle \mathcal M \mid
\Omega({\mathcal M}) \rangle$.  Consider the Cayley graph $X = Cay(G, \cal M)$ of $G$ with respect to
the generators ${\cal M}$. Assign to edges of $Cay(G, \cal M)$ their lengths
in $\Lambda$ and consider edges as closed intervals in $\Lambda$ of the corresponding length. For
each relation between bases of $\Omega$, $\lambda_{i_1} \cdots \lambda_{i_k} = \lambda_{j_1} \cdots
\lambda_{j_m}$ (without cancelation) there is a loop in $X$ labeled by this relation. Then the path
on the boundary of this loop labeled by $\lambda_{i_1} \cdots \lambda_{i_k}$ has the same length in
$\Lambda$ as the path labeled by $\lambda_{j_1} \cdots \lambda_{j_m}$. If $x$ is a point on the path
$\lambda_{i_1} \cdots \lambda_{i_k}$ at the distance $d \in \Lambda$ from the beginning of the path,
and $x'$ is a point on the path $\lambda_{j_1} \cdots \lambda_{j_m}$ at the distance $d$ from the
beginning, then we say that {\em $x$ and $x'$ are in the same leaf}. In other words, after we
substitute generators in $\mathcal M$ by their infinite word representations, we ``fold'' loops
into segments. We consider the equivalence
relation between points of $X$ generated by all such pairs $x \sim x'$. Equivalence classes of this
relation are called {\em leaves}. We also glue an arc isometric to a unit interval between each $x$
and $x'$. Let ${\cal F}$ be the foliation (the set of leaves). One can define a foliated complex
$\Sigma = \Sigma(X, {\cal F})$ associated to $X$ as a pair $(X, {\cal F})$. We define a metric space
$T(X)$ as follows. The paths in $\Sigma$ can travel vertically (along the leaves) and horizontally
(along the intervals in $\Lambda$). The length of a path $\gamma$ in $\Sigma (X, {\cal F})$ (denoted
$\|\gamma\|$) is the sum of the lengths of horizontal intervals in $\gamma$. The pseudo-distance $d$ between
two points of $\Sigma$ is defined as the length of a minimal path between them. A metric space $T$ is obtained from $\Sigma$ by identifying points at
pseudo-distance $0$. We extend naturally the left action of $G$ on $X$  to the action on $T$.

\begin{lemma}
\begin{enumerate}
\item[(1)] $T$ is a $\Lambda$-tree.
\item[(2)] The left action of $G$ on $T$ is free.
\end{enumerate}
\end{lemma}
\begin{proof} We will prove the second statement first. Suppose $x$ is a vertex of $X$, $g\in G$. Then $gx$ is also a vertex of $X$. If $d(gx, x)
= 0$ in $\Sigma$,  then $x$ and $gx$ are in the same leaf. Hence, there is a vertical path $x \sim x_1 \sim \cdots
\sim x_k = gx$ that uses relations from $\Omega({\mathcal M})$
$$R_1: p_{11}\ q_{11} = p_{12}\ \mu_1\ q_{12},\ \ \ R_2: p_{21}\ \mu_1\ q_{21} = p_{22}\ \mu_2\ q_{22},
\ \ \ \ldots,\ \ \ R_k: p_{k1}\ \mu_{k-1}\ q_{k1} = p_{k2}\ q_{k2},$$
where all $p_{ij}, q_{ij}$ are products of some bases ($G$ is generated by bases from ${\mathcal M}$);
$x$ is a vertex between $p_{11}$ and $q_{11}$, $x_i$ is a vertex on $\mu_i$, $x_k$ is a vertex between
$p_{k2}$ and $q_{k2}$. Then $\mu_i = \mu_{i1} \mu_{i2}$, where
$$\ell(p_{11}) = \ell(p_{12}\ \mu_{11}),\ \ell(p_{21}\ \mu_{11}) = \ell(p_{22}\ \mu_{21}), \ldots,
\ell(p_{i1}\ \mu_{(i-1) 1}) = \ell(p_{i2}\ \mu_{i1}),\ \ell(p_{k1}\ \mu_{(k-1) 1}) = \ell(p_{k2}).$$

Since relations are coming from the generalized equation $\Omega$, $p_{11}$ and $p_{12} \mu_{11}$ are
the same infinite words, as well as $p_{i1}\ \mu_{(i-1) 1}$ and $p_{i2}\ \mu_{i1}$ are the same infinite
words for $i = 2, \ldots, k-1$, and $p_{k1}\ \mu_{(k-1) 1}$ and $p_{k2}$ are the same infinite words.
Therefore $\mu_{11} = p_{12}^{-1}\ p_{11} = p_{21}^{-1}\ p_{22}\ \mu_{21},\ $ $\mu_{21} = p_{31}^{-1}\
p_{32}\ \mu_{31}, \ldots,\ $ $\mu_{(k-1) 1} = p_{k1}^{-1}\ p_{k2}$ and $p_{12}^{-1}\ p_{11} = p_{21}^{-1}\
p_{22} \ldots p_{k1}^{-1}\ p_{k2}$.

Then $p_{k2}^{-1}\ p_{k1} \cdots p_{12}^{-1}\ p_{11} = 1$ in $G$. But the element $p_{k2}^{-1}\ p_{k1} \cdots p_{12}^{-1}\ p_{11}\in G$ takes $gx\in X$
to $x\in X$, hence $ p_{k2}^{-1}\ p_{k1} \cdots
p_{12}^{-1}\ p_{11} gx = x$. Therefore $gx=x$ and $g = 1$.

If $x$ is not a vertex of $X$, then $x$ is a point on some edge of $X$ labeled by $\mu_0$ and joining
some $x_0$ and $x_0 \mu_0$, while $g x$, by translation in the Cayley graph, is a point on the edge
of $X$ labeled by $\mu_0$ and joining $g x_0$ and $g x_0 \mu_0$. There is a vertical path $x \sim x_1
\sim \cdots \sim x_k = gx$ that uses relations of $\Omega({\mathcal M})$
$$R_1: p_{11}\ \mu_0\ q_{11} = p_{12}\ \mu_1\ q_{12},\ \ \ R_2: p_{21}\ \mu_1\ q_{21} = p_{22}\ \mu_2\
q_{22},\ \ \ \ldots,\ \ \ R_k: p_{k1}\ \mu_{k-1}\ q_{k1} = p_{k2}\ \mu_0\ q_{k2}.$$
Let $y$ be a vertex of $X$ belonging to one of the loops $R_1, \ldots, R_k$ and nearest to the leaf
joining $x$ and $gx$. Suppose $y$ belongs to the loop $R_2$ and it is a vertex between $p_{21}$ and
$\mu_1$. Then there is a leaf joining $y$ and $gy$ that uses the relations $R_2, \ldots, R_k, R_1$,
and we can show that $g = 1$ as above.

\begin{figure}[htbp]
\centering{\mbox{\psfig{figure=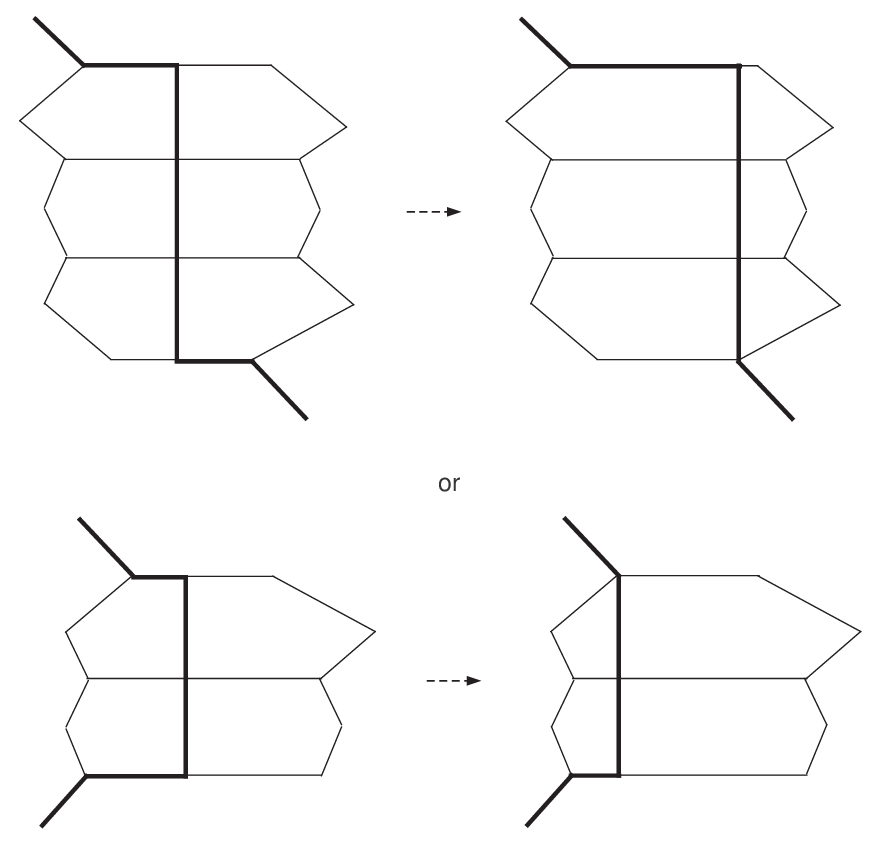,height=3in}}}
\caption{Replacing $\gamma_1$ by $\gamma_3$}
\label{pic:paths_0}
\end{figure}

Now we prove that $T$ is a $\Lambda$-tree. Suppose, we have a loop $\gamma$ in $\Sigma$ beginning
at $x$, which is a vertex of $X$ ($x \in VX$). We are going to show that the image of this loop becomes
degenerate in $T$ (if we remove all subpaths of the form $pp^{-1}$ from the image of $\gamma$ we get a point).
Notice that the image of any closed horizontal path in $\Sigma$ is degenerate in
$T$, so in the loop we can replace a horizontal subpath joining any two vertices $y,z \in VX$ by any
other horizontal path joining $y$ and $z$. Represent $\gamma = \gamma_1 \gamma_2$, where $\gamma_1$
is a sequence of vertical and horizontal paths between $x$ and another vertex $y$ of $X$, and
$\gamma_2$ is a horizontal path between $y$ and $x$ such that if there is a horizontal subpath $p$ at
the end of $\gamma_1$ then $p$ does not contain another vertex $z \in VX$. We can use induction on
the number of vertical components in $\gamma_1$, the case of zero vertical components being trivial.
We can replace the path $\gamma_1$ by a path $\gamma_3$ of equal or shorter length with the same
image in $T$ such that each vertical subpath of $\gamma_3$ is passing through some vertex of $X$
(see Figure \ref{pic:paths_0}).

\begin{figure}[htbp]
\centering{\mbox{\psfig{figure=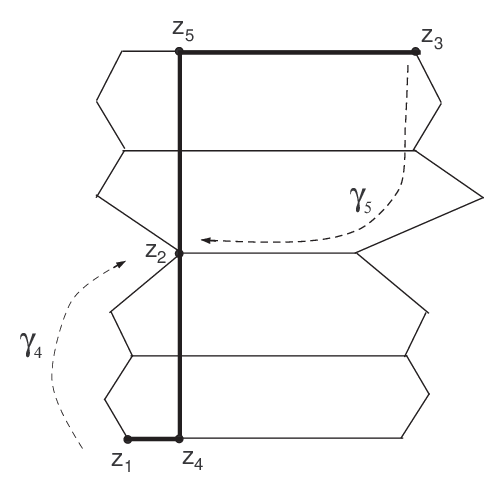,height=2.5in}}}
\caption{Getting rid of a vertical subpath}
\label{pic:paths}
\end{figure}

Let $z_1 \in VX$ be the first vertex of $\gamma _2$. Suppose the last vertical part of $\gamma_3$ begins
with $z_5$, ends with $z_4$ and is passing through $z_2 \in VX$. Then the image in $T$ of the circle
at $z_2$ consisting of the vertical path from $z_2$ to $z_4$ (see Figure \ref{pic:paths}) and the
corresponding horizontal path from $z_4$ to $z_2$ (denote it $\gamma _4$) is degenerate. The image
of the circle at $z_2$ consisting of the vertical path from $z_2$ to $z_5$ and the horizontal path
$z_2, z_3, z_5$ is also degenerate. Denote by $\gamma_5$ the horizontal path
between $z_3$ and $z_2$. We now replace $\gamma_2$ by $\gamma_5 \gamma_4^{-1} \gamma_2$ and $\gamma_1$
by the part of $\gamma_1$ from $x$ to $z_3$. The image of this path in $T$ is degenerate by induction
because it has less vertical subpaths. Therefore, $T$ is a $\Lambda$-tree.
\end{proof}

Similarly, we can construct a $\Lambda$-tree where $G$ acts freely beginning not with the original generalized equation $\Omega$ but with
a generalized equation obtained from $\Omega$ by the application of the Elimination process. In this case we may have closed sections with some items $h_i$ with $\gamma (h_i)=1. $ Then to construct the Cayley graph of $G$ in the new generators we cover items $h_i$ with $\gamma(h_i)
= 1$ by bases without doubles. (These bases were previously removed at some stages of the Elimination process
as matching pairs).

{\em Proof of Theorem \ref{th:main3}.} We begin by considering the union of the smallest closed
sections, see Remark \ref{fin}. Denote the union of these sections by $\sigma$. The group $H$ of the
generalized equation corresponding to their union is a free product of free groups, free abelian
groups and closed surface groups. Notice that for those sections for which $\gamma(h_i) = 1$ for some
maximal height items, these items $h_i$ are also products of some bases because initially every item
is a product of bases. We can assume the following:

\begin{enumerate}
\item For all closed sections of $\sigma$ such that $\gamma(h_i) = 2$ for all items of the maximal
height, the number of bases of maximal height cannot be decreased using entire transformation or
similar transformation applied from the right of the section ({\em right entire transformation});

\item  For all closed sections of $\sigma$ where we apply linear elimination, the number of bases of maximal height cannot be decreased using the transformations
as above as well as transformation (E2) (transfer) preceded  by creation of necessary boundary
connections as in (E5) (denote it ($E2_5$)), (E3) and first two types of linear elimination (D6)
(denote it ($D6_{1,2}$));

\end{enumerate}

Indeed, if we can decrease the number of bases of maximal height using the transformations described
above then we just do this and continue. Since these transformations do not change the total number
of bases we can also assume that they do not decrease the number of bases of second maximal height etc.
We now re-define the lengths of bases belonging to $\sigma$ in ${\mathbb R}^k$. We are going to show
that all components of the length of every base can be made zeros except for the components which
appear to be maximal in the lengths of bases from $\sigma$.

Let $\Omega_\sigma$ be the generalized equation corresponding to the sections from $\sigma$. Let $H$
be the group of the equation $H = \langle \mathcal M_\sigma \mid \Omega_\sigma(M) \rangle$.

Denote by $\Lambda_1$ the minimal convex subgroup of $\Lambda$ containing lengths of all bases in
$\sigma$, and by $\Lambda'$ a maximal convex subgroup of $\Lambda_1$ not containing lengths of maximal
height bases in $\sigma$ (it exists by Zorn's lemma). Then the quotient $\Lambda_1 / \Lambda'$ is a
subgroup of $\mathbb R$. Denote by $\hat\ell$ the length function in $\mathbb R$ on this quotient
induced from $\ell$. We consider elements of $\Lambda'$ as infinitesimals. Denote by $\widehat T$ the
$\mathbb R$-tree constructed from $T$ by identifying points at zero distance (see \cite{Ch1},
Theorem 2.4.7). Then $H$ acts on $\widehat T$. Denote by $\|\gamma\|_{\mathbb R}$ the induced length
of the path $\gamma$ in $\widehat T$, and by $d_{{\mathbb R}}(\bar x,\bar y)$ the induced distance.

However, the action of $H$ on $\widehat T$ is not free. The action is minimal, that is, there is no
non-empty proper invariant subtree. Notice that the canonical projection $f: T \rightarrow \widehat
T$ preserves alignment, and the pre-image of the convex set is convex. The pre-image of a point in
$\widehat T$ is an infinitesimal subtree of $T$.

\begin{lemma}
The action of $H$ on $\widehat T$ is superstable: for every non-degenerate arc $J \subset \widehat T$
with non-trivial fixator, and for every non degenerate subarc $S \subset J$, one has $Stab(S) =
Stab(J)$.
\end{lemma}

The proof is the same as the proof of Fact 5.1 in \cite{G}.

\begin{lemma}
\label{quad-main}
Let $\sigma_1$ be a closed section on the lowest level corresponding to a closed surface group and
the lengths of the bases satisfy some system of linear equations $\Sigma$. Then one can define the
lengths of bases in $\sigma_1$ in ${\mathbb R}^k$.
\end{lemma}
\begin{proof} Denote by ${\mathcal M}_{\sigma +}$ the set of bases (on all the steps of the process
of entire transformation applied to the lowest level) with non-zero oldest component and by
${\mathcal M}_{\sigma 0}$ the rest of the bases (infinitesimals). Denote by $\hat\ell$ the projection
of the length function $\ell$ to $\Lambda_1 / \Lambda'$ as before. Then $\lambda \in
{\mathcal M}_{\sigma +}$ if and only if $\hat\ell(\lambda) > 0$. We apply the entire transformation
to $\sigma_1$. If we obtain an overlapping pair or an infinitesimal section, where the process goes
infinitely, we declare it non-active and move to the right. This either decreases, or does not change
the number of bases in the active part. Therefore, we can assume that the process goes infinitely
and the number of bases in ${\mathcal M}_{\sigma +}$ never decreases. Therefore bases from
${\mathcal M}_{\sigma 0}$ are only used as transfer bases.

We will show that the stabilizer of a pre-image of a point, an infinitesimal subtree $T_0$ of $T$ is
generated by some elements in ${\mathcal M}_{\sigma 0}$. An element $h$ from $H$ belongs to such a
stabilizer if $\hat\ell(h) = 0$. If some product of bases not only from ${\mathcal M}_{\sigma 0}$
has infinitesimal length (denote this product by $g$), then the identity and $g$ belong to leaves at
the infinitesimal distance $\delta$ in $\Sigma$. Therefore using elementary operations we can obtain
a base of length $\delta$. Denote by $\Lambda''$ the minimal convex subgroup of $\Lambda$ containing
all elements of the same height as bases from ${\mathcal M}_{\sigma 0}$.

This implies the following lemma which we need to finish the proof of lemma \ref{quad-main}.

\begin{lemma}
\label{quad}
Let $\overline\sigma_1$ be the projection of the quadratic section $\sigma_1$ to $\Lambda_1 /
(\Lambda'' \cap \Lambda_1)$. Suppose the process of entire transformation for $\sigma_1$ goes
infinitely and the number of bases in ${\mathcal M}_{\sigma +}$ never decreases. Then the process
of entire transformation for $\overline\sigma_1$ goes infinitely too and the number of bases in
${\mathcal M}_{\sigma +}$ never decreases.

If $\ell(g) \in \Lambda'$ then $g$ is a product of bases in ${\mathcal M}_{\sigma 0}$ and $\ell(g)
\in \Lambda''$.
\end{lemma}

This lemma implies that there is no element in $H$ infinitely larger than all bases in
${\mathcal M}_{\sigma 0}$, but infinitely smaller than all bases in ${\mathcal M}_{\sigma +}$.

Therefore, we can make all components of the lengths of bases in ${\mathcal M}$ zeros except for
those which are maximal components of some bases in ${\mathcal M}$. Since ${\mathcal M}$ is a finite
set, the number of such components is finite, and the length is defined in ${\mathbb R}^k$ for some
$k$ not larger than the number of pairs of bases.
\end{proof}

Let $\sigma$ be a closed section corresponding to infinite linear elimination process for a free
group as in Lemma \ref{lin-el} (c). We cover (for technical reasons) items $h_i$ with $\gamma(h_i)
= 1$ by bases without doubles. (These bases were removed at some stages of the Elimination process
as matching pairs).

\begin{lemma}
\label{lin}
Let $\sigma$ be a closed section corresponding to infinite linear elimination process for a free
group as in Lemma \ref{lin-el} (c). Let $\overline\sigma$ be the projection of section $\sigma$
to $\Lambda_1 / \Lambda''$.  Suppose the process of linear elimination on $\sigma $ goes infinitely
and some generalized equation $\Omega'$ (with the minimal possible number of bases) appears in this
process infinitely many times. Then the process of linear elimination on $\overline\sigma$ goes
infinitely. Moreover, suppose for any base $\mu$ with $\hat\ell(\mu) = 0$ obtained by elementary
transformations, $\ell(\mu) \in \Lambda''$. If $H = \langle {\mathcal M}_{\sigma} \mid \Omega_\sigma
({\mathcal M}_\sigma) \rangle$, and $g \in H, \hat\ell(g) = 0$ then $\ell(g) \in \Lambda''$.
\end{lemma}
\begin{proof} The first statement holds because the steps that we make in the process depend only on
the situation of bases of maximal height. If some product $g$ of bases not only from
${\mathcal M}_{\sigma 0}$ has infinitesimal length, then the identity and $g$ belong to leaves at
the infinitesimal distance $\delta$ in $\Sigma$. Therefore using elementary operations we can obtain
a base $\mu$ of length comparable to $\delta$. Therefore $\delta \in \Lambda''$.
\end{proof}

\begin{lemma}
\label{lin1}
Let $\overline \sigma$ be a closed section corresponding to infinite linear elimination process for
a free group.  Suppose the process of linear elimination on $\overline\sigma$ goes infinitely and
some generalized equation $\Omega$ (with the minimal possible number of bases) appears in this
process infinitely many times. Suppose all items have length in ${\mathbb R}$. Then applying
elementary transformations we cannot decrease the difference $k$ between the number of items and
the total number of basic and boundary equations. This difference is the rank of the free group
that is the coordinate group of $\Omega$.
\end{lemma}
\begin{proof} Every equation $\Omega '$ that is obtained from $\Omega$ by elementary operations
has the free group of rank $k$ as the coordinate group. If we cut bases of $\Omega'$ along boundary
connections and start removing eliminable bases, we have to remove them all. Every time when we
remove eliminable base we express one item in terms of the others. There must be $k$ items left at
the end.
\end{proof}

Lemma \ref{lin1} implies that if $\sigma$ is a closed section corresponding to infinite linear
elimination process for a free group, then a base  $\mu$ with $\hat\ell(\mu) = 0$ can only be
obtained by elementary transformations from bases in ${\cal M}_{\sigma 0}$. Therefore, by Lemma
\ref{lin}, if $H = \langle {\mathcal M}_{\sigma} \mid \Omega_\sigma ({\mathcal M}_\sigma) \rangle$,
and $g \in H, \hat\ell(g) = 0$ then $\ell(g) \in \Lambda''$.

Finally, if a finitely generated abelian group $H$ acts freely on a $\Lambda$-tree, then it is
isometrically embedded into $\Lambda$ (see, for example, \cite{B}).
\begin{lemma}(\cite{Ch1}, Lemma 1.1.6) If $H$ is a finitely generated ordered abelian group, it has a decomposition $H =\Lambda _1\bigoplus\ldots\bigoplus\Lambda _k$ for some $k\geq 0$, where the ordering is the lexicographic ordering and $\Lambda _1,\ldots \Lambda _k$ are subgroups of ${\mathbb R}.$
\end{lemma}
Therefore, $H$ has a free basis $h_1,\ldots h_k$ such that $ht (h_1)>\ldots ht (h_k)$. We can now make all the components in the length of the basis elements zeros except those corresponding to $ht (h_1),\ldots , ht (h_k)$. This defines a
free length function in ${\mathbb R}^n$ on $H$. The new lengths of elements in $H$ satisfy the same system of linear equations as their lengths in
$\Lambda$.

Using induction on the number of levels obtained in the Elimination process, we  similarly prove
the statement of Theorem \ref{th:main3}.

The points of an ${\mathbb R}^n$-tree, where $G$ acts freely are the leaves in the foliation
corresponding to the new length of bases in ${\mathbb R}^n$. The new lengths of bases are exactly
their Lyndon lengths. $\Box$

\smallskip

{\em  Proof of Theorem \ref{th:main4}}. Notice that in the case when $\widetilde G$ is a finitely
presented group with a free length function in $\Lambda$ (not necessary regular) it can be embedded
in the group with a free regular length function in $\Lambda$ by Remark \ref{nonreg}. That group can
be embedded in $R(\Lambda', X)$. When we make a generalized equation for $\widetilde G$, we have to
add only a finite number of elements from $R(\Lambda', X)$. We run the elimination process for this
generalized equation as we did in the proof of Theorem \ref{th:main1} and obtain a group $G$ as in
Remark \ref{nonreg1}, where $\widetilde G$ is embedded, and then redefine the length of elements of
$G$ in ${\mathbb R}^n$ as above. Therefore $G$ acts freely and regularly on a ${\mathbb R}^n$-tree.
The theorem is proved.

Moreover, using Lemmas \ref{quad} and \ref{lin}, and induction one shows that the length function in
${\mathbb R}^n$ and in $\Lambda$ defined on $G$ is regular. Therefore we proved Theorem \ref{th:main4}.

\begin{remark}
If $\widetilde G$ is a finitely presented group which has a free (not necessary regular) length
function in $\Lambda$, then it acts freely on a $\Lambda$-tree, and one can construct, as in \cite{BF},
a foliated band complex with measured foliation with horizontal measure in $\Lambda$.
\end{remark}


\begin{thebibliography}{99}

\bibitem{AB}
R. Alperin and H. Bass,
{\it Length functions of group actions on $\Lambda$-trees.}
Combinatorial group theory and topology, (Ed. S. M. Gersten and J. R. Stallings),
Annals of Math. Studies \textbf{111}, 265--378. Princeton University Press, 1987.

\bibitem{AM}
R. Alperin and K. Moss,
{\it Complete trees for groups with a real length function.}
J. London Math. Soc. (2) \textbf{31}, 1985, 55--68.

\bibitem{B}
H. Bass
{\it Groups acting on non--archimedian trees.}
Arboreal group theory, 1991, 69--130.

\bibitem{BF}
M. Bestvina and M. Feighn,
{\it Stable actions of groups on real trees.}
Invent. Math., \textbf{121} no.~2 (1995), 287--321.

\bibitem{bridson}
M. R. Bridson and A. Haefliger,
{\it Metric spaces of non-positive curvature.}
Springer, 1999.

\bibitem{Ch}
I. Chiswell,
{\it Abstract length functions in groups.}
Math. Proc. Cambridge Philos. Soc., \textbf{80} no.~3 (1976), 451--463.

\bibitem{Ch1}
I. Chiswell,
{\it Introduction to $\Lambda$-trees.}
World Scientific, 2001.

\bibitem{Ch3}
I. Chiswell,
{\it $A$-free groups and tree-free groups.}
{\it Algorithms, Languages, Logic} (Ed. A. Borovik),
Contemp. Math., Amer. Math. Soc. \textbf{378} (2005), 79--86.

\bibitem{Ch2}
I. Chiswell and T. Muller,
{\it Embedding theorems for tree-free groups.}
Under consideration for publication in Math. Proc. Camb. Phil. Soc.

\bibitem{Dah}
F. Dahmani,
{\it Existential questions in (relatively) hyperbolic groups.}
Israel J. Math \textbf{173} (2009), 91--124.

\bibitem{DG}
F. Dahmani and D. Groves,
{\it The Isomorphism Problem for Toral Relatively Hyperbolic Groups.}
Publ. Math., Inst. Hautes Etudes Sci. \textbf{107} no.~1 (2008), 211--290.

\bibitem{DP}
T. Delzant and L. Potyagailo,
{\it Accessibilit$\acute{e}$ hi$\acute{e}$rarchique des groupes de pr$\acute{e}$sentation finie.}
(French) Topology \textbf{40} no.~3 (2001), 617--629.

\bibitem{GLP}
D. Gaboriau, G. Levitt and F. Paulin,
{\it Pseudogroups of isometries of $\mathbb{R}$ and Rips' Theorem on free actions on $\mathbb{R}$-trees.}
Israel. J. Math., \textbf{87} (1994), 403--428.

\bibitem{GabLev}
D. Gaboriau and G. Levitt and F. Paulin,
{\it Pseudogroups of isometries of $\mathbb{R}$: reconstruction of free actions on $\mathbb{R}$-trees.}
Ergodic Theory Dynam. Systems \textbf{15} (1995), 633--652.

\bibitem{GKM}
D. Gildenhuys, O. Kharlampovich and A. Myasnikov,
{\it CSA-groups and separated free constructions.}
Bull. Austral. Math. Soc., \textbf{52} no.~1 (1995), 63--84.

\bibitem{Gl}
A. M. W. Glass,
{\it Partially ordered groups.}
Series in Algebra, \textbf{7}, World Scientific, 1999.

\bibitem{Gr}
D. Groves,
{\it Limit groups for relatively hyperbolic groups, II: Makanin-Razborov diagrams.}
Geom. Topol. \textbf{9} (2005), 2319-2358.

\bibitem{G}
V. Guirardel,
{\it Limit groups and groups acting freely on $\mathbb{R}^n$-trees.}
Geom. Topol., \textbf{8} (2004), 1427--1470.

\bibitem{Harrison}
N. Harrison,
{\it Real length funtions in groups.}
Trans. Amer. Math. Soc. \textbf{174} (1972), 77--106.

\bibitem{Hoare1}
A. H. M. Hoare,
{\it On length functions and Nielsen methods in free groups.}
J. London Math. Soc. (2) \textbf{14} (1976), 188--192.

\bibitem{Hoare2}
A. H. M. Hoare,
{\it Nielsen method in groups with a length function.}
Math. Scand. \textbf{48} (1981), 153--164.

\bibitem{KhMS}
B. Khan, A. Myasnikov and D. Serbin,
{\it On positive theories of groups with regular free length functions.}
Internat. J. Algebra and Comput. \textbf{17} no.~1 (2007), 1--26.

\bibitem{KMIrc}
O. Kharlampovich and A. Myasnikov,
{\it Irreducible affine varieties over a free group. II: Systems in triangular
quasi-quadratic form  and description of residually free groups.}
J. of Algebra \textbf{200} no.~2 (1998), 517--570.

\bibitem{KM2}
O. Kharlampovich and A. Myasnikov,
{\it Implicit function theorems over free groups.}
J. of Algebra \textbf{290} no.~1 (2005), 1--203.

\bibitem{KMRS1}
O. Kharlampovich, A. Myasnikov, V. Remeslennikov and D. Serbin
{\it Groups with free regular length functions in $\mathbb{Z}^n$.}
ArXiv:0907.2356, Trans. Amer. Math. Soc., in print.

\bibitem{KMS2}
O. Kharlampovich, A. Myasnikov and D. Serbin,
{\it Regular completions of $\mathbb{Z}^n$-free groups.}
Preprint, 2011.

\bibitem{KopMed}
V. Kopytov and N. Medvedev,
{\it Right-ordered groups.}
Siberian School of Algebra and Logic. Consultants Bureau, New York, 1996.

\bibitem{L}
R. Lyndon,
{\it Length functions in groups.}
Math. Scand. \textbf{12} (1963), 209--234.

\bibitem{LS}
R. C. Lyndon and P. E. Schupp,
{\it Combinatorial group theory.}
Ergebnisse der Mathematik und ihrer Grenzgebiete \textbf{89},
Springer-Verlag, Berlin, Heidelberg, New York, 1977.

\bibitem{Mak}
G. S. Makanin,
{\it Equations in a free group.}
(Russian), Izv. Akad. Nauk SSSR, Ser. Mat., \textbf{46} (1982),
1199--1273 (transl. in Math. USSR Izv. \textbf{21} (1983)).

\bibitem{MR}
A. Myasnikov and V. Remeslennikov,
{\it Length functions on free exponential groups.}
Proc. $N$ 26. IITPM SO RAN, Omsk, 1996, 1--34.

\bibitem{MRS}
A. Myasnikov, V. Remeslennikov and D. Serbin,
{\it Regular free length functions on Lyndon's free $\mathbb{Z}[t]$-group
$F^{\mathbb{Z}[t]}$,}
{\it Algorithms, Languages, Logic} (Ed. A. Borovik),
Contemp. Math., Amer. Math. Soc. \textbf{378} (2005), 37--77.

\bibitem{MS}
J. Morgan and P. Shalen,
{\it Valuations, Trees, and Degenerations of Hyperbolic Structures, I.}
Annals of Math, 2nd Ser., \textbf{120} no.~3. (1984), 401--476.

\bibitem{Prom}
D. Promislow,
{\it Equivalence classes of length functions on groups.}
Proc. London Math. Soc (3) \textbf{51} (1985), 449--477.

\bibitem{Raz}
A. Razborov,
{\it On systems of equations in a free group.}
Math. USSR Izvestiya \textbf{25} no.~1 (1985), 115--162.

\bibitem{Serre}
J.-P. Serre,
{\it Trees.}
New York, Springer, 1980.

\bibitem{AS}
A. Schrijver,
{\it Theory of linear and integer programming.}
New-York, John Wiley and Sons, 1998.

\bibitem{Rebeca}
D. Y. Rebbechi,
{\it Algorithmic properties of relatively hyperbolic groups},
PhD Thesis, Univ. California, Davis, 2001. http://front.math.ucdavis.edu/math.GR/0302245.

\bibitem{wise} D. Wise The structure of groups with a quasiconvex hierarchy, preprint
\end{thebibliography}
\end{document}